\documentclass[11pt,a4paper]{amsart}
\usepackage[colorlinks, linkcolor=red, anchorcolor=blue, citecolor=green]{hyperref}
\usepackage{amsthm,bm} 
\usepackage{amsfonts}
\usepackage{amsmath}
\usepackage{mathrsfs}
\usepackage{amssymb}
\usepackage{amsfonts}
\usepackage{amsbsy}
\usepackage{ifthen}
\usepackage[all]{xy}
\usepackage{extarrows}
\usepackage{CJK,CJKnumb}
\usepackage{color}              
\usepackage{indentfirst}        
\usepackage{latexsym,bm}        
\usepackage{graphicx}
\usepackage{cases}
\usepackage{pifont}
\usepackage{txfonts}
\usepackage{xcolor}
\usepackage{multirow}

\newcounter{maint}

\newtheorem{mteor}[maint]{Theorem}
\allowdisplaybreaks[4]  

\CJKtilde   

\numberwithin{equation}{section}

\newcommand{\ydh}{{}^H_H\mathcal{YD}}
\newcommand{\mi}{\mathrm{i}}

\begin{document}

\newtheorem{theorem}{Theorem}[section]

\newtheorem{lemma}[theorem]{Lemma}

\newtheorem{corollary}[theorem]{Corollary}
\newtheorem{proposition}[theorem]{Proposition}

\theoremstyle{remark}
\newtheorem{remark}[theorem]{Remark}

\theoremstyle{definition}
\newtheorem{definition}[theorem]{Definition}

\theoremstyle{definition}
\newtheorem{conjecture}[theorem]{Conjecture}

\newtheorem{example}[theorem]{Example}


\title[ Finite dimensional Hopf algebras over Kac-Paljutkin algebra
$H_8$]{ Finite dimensional Hopf algebras over Kac-Paljutkin algebra $H_8$}

\author[Shi]{Yuxing Shi }
\address{School of Mathematics and Information Science, Guangzhou University,  Guangzhou 510006, P.R. China}\email{blueponder@foxmail.com}

%



\begin{abstract}
Let $H_8$ be the neither commutative nor cocommutative  semisimple  eight dimensional Hopf algebra, which is also called Kac-Paljutkin algebra \cite{MR0208401}. All simple Yetter-Drinfel'd modules over $H_8$ are given. As for simple objects and  direct sums of two simple objects in ${}_{H_8}^{H_8}\mathcal{YD}$,  we calculated dimensions for  the corresponding Nichols algebras, except four  semisimple cases which are generally difficult. Under the assumption that the four undetermined Nichols algebras are all infinite dimensional, we determine all the finite dimensional Nichols algebras over $H_8$. It turns out that  the already known finite dimensional Nichols algebras  are all diagonal type. In fact, they are Cartan types $A_1$, $A_2$, $A_2\times A_2$, $A_1\times \cdots \times A_1$, and $A_1\times \cdots \times A_1\times A_2$. By the way, we calculate Gelfand-Kirillov dimensions for some Nichols algebras. As an application, we obtain five families of  new  finite dimensional  Hopf algebras over $H_8$ according to the lifting method.
\end{abstract}
\maketitle

\section{Introduction}
Let $\mathbb{K}$ be an algebraically closed field of characteristic zero. The question of classification of all Hopf algebras over $\mathbb{K}$ of a given dimension
up to isomorphism was posed by Kaplansky in 1975 \cite{Kaplansky1975MR0435126}. 
Some progress has been made but, in general, it is
a difficult question for lack of standard methods.
One breakthrough is the so-called \emph{Lifting Method} \cite{MR1659895} introduced by Andruskiewitsch and Schneider in 1998,
under the assumption that the coradical is a Hopf subalgebra. 

We describe the procedure for the lifting method briefly. Let $H$ be a Hopf algebra whose coradical $H_0$ is a Hopf subalgebra. It is well-known that the associated
graded Hopf algebra of $H$ is isomorphic to $R\#H_0$ where
$R = \oplus_{n\in \mathbb N_0}R(n)$ is a braided Hopf algebra in the category ${}_{H_0}^{H_0}\mathcal{YD}$ of Yetter-Drinfield modules over $H_0$. $\#$ stands for the Radford biproduct or \textit{bosonization} of
$R$ with $H_0$. As explained in \cite{andruskiewitsch2001pointed}, to classify finite-dimensional Hopf algebras $H$ whose coradical is isomorphic to $H_0$ we have to deal with the following questions:
\begin{enumerate}\renewcommand{\theenumi}{\alph{enumi}}\renewcommand{\labelenumi}{(\theenumi)}
  \item\label{que:nichols-fd} Determine all Yetter-Drinfiel'd modules $V$ over $H_0$ such that the Nichols algebra $\mathfrak{B}(V)$ has finite dimension; find an efficient set of relations for $\mathfrak{B}(V)$.
\item\label{que:nichols-R} If $R = \oplus_{n\in \mathbb N_0}R(n)$ is a finite-dimensional Hopf algebra in ${}_{H_0}^{H_0}\mathcal{YD}$ with $V = R(1)$, decide if $R \simeq \mathfrak{B}(V)$. Here $V = R(1)$ is a braided vector space called the \textit{infinitesimal braiding}.
\item\label{que:lifting} Given $V$ as in \eqref{que:nichols-fd}, classify all $H$ such that $\mathrm{gr}\, H \simeq \mathfrak{B}(V)\#H_0$ (lifting).
\end{enumerate}

The lifting method was extensively used in the classification of finite dimensional pointed Hopf algebras such as \cite{andruskiewitsch2005classification},  \cite{andruskiewitsch2010nichols}, \cite{MR2811166}\cite{MR2862142},  \cite{andruskiewitsch2011finite}, \cite{andruskiewitsch2010pointed}, \cite{MR3493214}, \cite{MR3395052} and so on. It is also effective to study finite-dimensional copointed Hopf algebras \cite{MR2863448}, \cite{MR3119229}, \cite{2016arXiv160806167F}. We note that there are very few classification results on finite-dimensional Hopf algebras whose coradical is a Hopf subalgebra but not a group algebra and the dual of a group algebra, two exceptions being  \cite{MR2037722, 2016arXiv160503995A}.

Here we would like to initiate a project for the study of Hopf algebras whose coradicals are low-dimensional neither commutative nor cocommutative semisimple Hopf algebras by running  procedures of the lifting method. One important step is to study the Nichols algebras over those low-dimensional semisimple Hopf algebras. Nichols algebras were studied  first by Nichols \cite{nichols1978bialgebras} . These are connected graded braided Hopf algebras \cite{MR1907185} generated by primitive elements, and all primitive elements are of degree one. In the past decades,  the study of Nichols algebras was mainly focused on group algebras and which were finite dimensional, for those Nichols algebras were  essential ingredients of the classification of finite-dimensional pointed Hopf algebras. Under the assumption that the base field has characteristic $0$, the classification of finite-dimensional Nichols
algebras over abelian groups was completely solved in \cite{MR2207786, heckenberger2009classification} by using Lie theoretic structures, and the result of the classification played an important role later in the significant work \cite{andruskiewitsch2005classification}. The problem of classifying finite-dimensional Nichols algebras over non-abelian  groups is difficult in general for lack of systematic method, related works please refer to \cite{andruskiewitsch2010nichols}, \cite{freyre2007nichols},\cite{grana2011nichols}, 
\cite{MR2891215}
\cite{MR2732989},
\cite{MR3276225},
\cite{MR3272075}, etc.

In this paper, we mainly focus on  Kac-Paljutkin algebra $H_8$. The structure of our paper is as follows.
In Section \ref{Preliminaries},  we recall the fundamental
notions related to Yetter-Drinfel'd modules, Nichols algebras and 
Gelfand-Kirillov dimension.  In section \ref{YDMH8}, we construct all the simple left Yetter-Drinfel'd modules over $H_8$ according to Radford's method. In section \ref{section:NicholsAlgebras}, we get all the possible finite dimensional Nichols algebras from Yetter-Drinfel'd modules over $H_8$ under the assumption that  the four undetermined cases over semisimple modules  for which are difficult for us at the moment.  It turns out that all the already known finite dimensional Nichols algebras are Cartan types $A_1$, $A_2$, $A_2\times A_2$, $A_1\times \cdots \times A_1$, and $A_1\times \cdots \times A_1\times A_2$. Here is our first main result.
\begin{mteor}
\label{NicholsAlg:maintheorem}
 Suppose 
 $$\dim\mathfrak{B}\left(M\langle(xy,x)\rangle
\oplus W^{b_1,-1}\right)=\infty
=\dim\mathfrak{B}\left(M\langle(y,xy)\rangle
\oplus W^{b_1,-1}\right)$$ 
holds for $b_1=\pm 1$.
 If $M\in {}_{H_8}^{H_8}\mathcal{YD}$ such that $\dim\mathfrak{B}(M)<\infty$, then $M$ is isomorphic either to one of the following Yetter-Drinfel'd modules
 \begin{enumerate}
 \item $\Omega_1(n_1,n_2,n_3,n_4)\triangleq
 \bigoplus_{j=1}^4 M\left<1|V_1(b_j),g_j\right>^{\oplus n_j}$
 with $\sum_{j=1}^4 n_j\geq 1$, $(b_1,g_1)=(\mi,x)$, $(b_2,g_2)=(-\mi,x)$, $(b_3,g_3)=(\mi,y)$, $(b_4,g_4)=(-\mi,y)$,
the infinitesimal braiding is type $\underbrace{A_1\times \cdots \times A_1}_{n_1+n_2+n_3+n_4}$.
 \item $\Omega_2(n_1,n_2)\triangleq M\langle\mi,x\rangle^{\oplus n_1}
 \oplus M\langle -\mi,x\rangle^{\oplus n_2}
 \oplus M\langle(xy,x)\rangle$
 , $n_1+n_2\geq 0$, the infinitesimal braiding is type 
  $\underbrace{A_1\times \cdots \times A_1}_{n_1+n_2}\times\, A_2$.
 \item $\Omega_3(n_1,n_2)\triangleq M\langle\mi,y\rangle^{\oplus n_1}
 \oplus M\langle -\mi,y\rangle^{\oplus n_2}
 \oplus M\langle(y,xy)\rangle$
 , $n_1+n_2\geq 0$, the infinitesimal braiding is type 
  $\underbrace{A_1\times \cdots \times A_1}_{n_1+n_2}\times\, A_2$.
 \item $\Omega_4(n_1,n_2)\triangleq M\langle\mi,x\rangle^{\oplus n_1}
 \oplus M\langle\mi,y\rangle^{\oplus n_2}
 \oplus W^{1,-1}$
 , $n_1+n_2\geq 0$, the infinitesimal braiding is type 
  $\underbrace{A_1\times \cdots \times A_1}_{n_1+n_2}\times\, A_2$.
 \item $\Omega_5(n_1,n_2)\triangleq M\langle -\mi,x\rangle^{\oplus n_1}
 \oplus M\langle -\mi,y\rangle^{\oplus n_2}
 \oplus W^{-1,-1}$
 , $n_1+n_2\geq 0$, the infinitesimal braiding is type 
  $\underbrace{A_1\times \cdots \times A_1}_{n_1+n_2}\times\, A_2$.
  \item $\Omega_6\triangleq M\langle(xy,x)\rangle
 \oplus M\langle(y,xy)\rangle$, the infinitesimal braiding is type $A_2\times A_2$.
 \item $\Omega_7\triangleq W^{1,-1}\oplus W^{-1,-1}$, the infinitesimal braiding is type $A_2\times A_2$.
 \end{enumerate}
\end{mteor}
\begin{remark}
We point out which of the Yetter-Drinfeld modules  have a principal realization and which not, since the liftings are known when there is a principal realization and not otherwise \cite{AAGI}.
Let $(h)$ and $(\delta_h)$ be a dual basis of $H_8$ and $H_8^{*}$, and  $b\in\{\pm 1, \pm\mi\}$,  then $$\chi_b:=\delta_1+\delta_{xy}+b(\delta_x+\delta_y)+
b^2(\delta_{z}+\delta_{zxy})+b^3(\delta_{zx}+\delta_{zy})\in\mathrm{Alg}(H_8,\mathbb{K}).$$
$(g,\chi_b)$ is a \emph{YD-pair} \cite{MR3133699} and 
$\mathbb{K}_g^{\chi_b}\simeq M\langle b,g\rangle$. $M\langle(g_1,g_2)\rangle$ for $(g_1,g_2)\in \{(xy,x),(y,xy)\}$ and $W^{b_1,-1}$ for $b_1=\pm 1$ don't have a \emph{principal realization} \cite[Subsection 2.2]{AAGI},  since $\mathbb{K}_{g_1}^{\chi_{b_1}}\oplus 
\mathbb{K}_{g_2}^{\chi_{b_2}}$ is of type $A_1\times A_1$ for 
$(b_1,g_1), (b_2,g_2)\in\{(\pm \mi,x), (\pm \mi,y)\}$.
So only $\Omega_1(n_1,n_2,n_3,n_4)$ has a principal realization.
\end{remark}

In section \ref{section:HopfAlgebras}, according to the lifting method, we give a classification for finite-dimensional Hopf algebras over $H_8$ such that  their infinitesimal braidings  are  isomorphic to those Yetter-Drinfel'd modules listed in Theorem \ref{NicholsAlg:maintheorem}. Here is the second main result. 

\begin{mteor}\label{HopfAlgOverH8}
 Suppose  $H$ is a finite-dimensional Hopf algebra over $H_8$ such that its 
 infinitesimal braiding is isomorphic to one of the Yetter-Drinfel'd modules listed in  Theorem \ref{NicholsAlg:maintheorem},  then  $H$ is isomorphic either to
 \begin{enumerate}
 \item $\mathfrak{A}_1(n_1,n_2,n_3,n_4;I_1)$, see Definition \ref{definition:HopfAlgA_1};
 \item $\mathfrak{B}[\Omega_2(n_1, n_2)]\# H_8$, see Proposition 
 \ref{HopfAlg:Omega2};
 \item $\mathfrak{A}_4(n_1,n_2;I_4)$, see Definition \ref{Definition:HopfAlgA_4};
 \item $\mathfrak{A}_6(\lambda)$, see Definition \ref{definition:A_6};
 \item $\mathfrak{A}_7(I_7)$, see Definition \ref{Definition:HopfAlgA_7}.
 \end{enumerate}
\end{mteor}
$\mathfrak{A}_1(n_1,n_2,n_3,n_4;I_1)$ comprises two nonisomorphic nonpointed self-dual Hopf algebras of dimension
$16$ with coradical $H_8$ described  in
\cite{MR2037722} as special cases.  Except $\mathfrak{B}[\Omega_2(n_1, n_2)]\# H_8$, the remainder four families of Hopf algebras contain non-trial lifting relations.

 \section{Preliminaries}\label{Preliminaries}
 \subsection{Conventions}
Let $H$ be a Hopf algebra over $\mathbb{K}$, with antipode 
$S$. We will use Sweedler's notation $\Delta(h)=h_{(1)}\otimes 
h_{(2)}$ for the comultiplication \cite{montgomery1993hopf}. Let ${}_{H}^{H}\mathcal{YD}$  be the category of left \emph{Yetter-Drinfel'd 
modules} over $H$. That is to say that if $M$ is an object of 
 ${}_{H}^{H}\mathcal{YD}$ if and only if there exists 
 an action $\cdot$ such that $(M,\cdot)$ is a left 
 $H$-module and a coaction  $\rho$ such that $(M,\rho)$
 is a left $H$-comodule, subject to the following compatibility condition:
 \begin{equation}
 \rho(h\cdot m)=h_{(1)}m_{(-1)}S
 (h_{(3)})\otimes h_{(2)} \cdot m_{(0)}, \forall m\in M, h\in H,
 \end{equation}
 where $\rho(m)=m_{(-1)}\otimes m_{(0)}$. It is a   braided monoidal category. 
The braiding $c\in \mathrm{End}_\mathbb{K}(M\otimes M)$ of $M$ is defined by 
$ c(v\otimes w)=v_{(-1)} \cdot w\otimes v_{(0)}$, and the inverse braiding is defined by 
$c^{-1}(v\otimes w)=w_{(0)}\otimes \left(S^{-1}(w_{(-1)})\cdot v\right)$.

\begin{definition}\cite[Definition. 2.1]{andruskiewitsch2001pointed}
\label{defNicholsalgebra}
Let $H$ be a Hopf algebra and $V \in \ydh$. A braided $\mathbb{N}$-graded 
Hopf algebra $R = \bigoplus_{n\geq 0} R(n) \in \ydh$  is called 
the \textit{Nichols algebra} of $V$ if 
\begin{enumerate}
 \item[(i)] $\mathbb{K}\simeq R(0)$, $V\simeq R(1) \in \ydh$,
 \item[(ii)] $R(1) = \mathcal{P}(R)
=\{r\in R~|~\Delta_{R}(r)=r\otimes 1 + 1\otimes r\}$.
 \item[(iii)] $R$ is generated as an algebra by $R(1)$.
\end{enumerate}
In this case, $R$ is denoted by $\mathfrak{B}(V) = \bigoplus_{n\geq 0} \mathfrak{B}^{n}(V) $.    
\end{definition}
\begin{remark}
The Nichols algebra 
$\mathfrak{B}(V)$ is completely determined by the braiding.
Let $\mathfrak{B}(M)$ denote the Nichols algebra generated by $M\in {}_{H}^{H}\mathcal{YD}$. More precisely, as proved in  \cite{MR1396857} and
noted in \cite{andruskiewitsch2001pointed},
$$\mathfrak{B}(M)=\mathrm{K}\oplus M\oplus\bigoplus\limits_{n=2}^\infty M^{\otimes n} /
\ker\mathfrak{S}_n=T(M)/\ker\mathfrak{S},$$
where $\mathfrak{S}_{n,1}\in \mathrm{End}_\mathbb{K}\left(M^{\otimes (n+1)}\right)$, 
$\mathfrak{S}_{n}\in \mathrm{End}_\mathbb{K}\left(M^{\otimes n}\right)$,
$$\mathfrak{S}_{n,1}\coloneqq\mathrm{id}+c_n+c_nc_{n-1}+\cdots+c_nc_{n-1}\cdots c_1=\mathrm{id}+c_n\mathfrak{S}_{n-1,1}$$
$$\mathfrak{S}_1\coloneqq\mathrm{id}, \quad \mathfrak{S}_2\coloneqq\mathrm{id}+c, \quad
\mathfrak{S}_n\coloneqq(\mathfrak{S}_{n-1}\otimes \mathrm{id})\mathfrak{S}_{n-1,1}.$$
\end{remark}
%
\begin{lemma}\label{TensorNicholsAlg}
(\cite[Theorem 2.2]{MR1779599}, \cite[Remark 1.4]{andruskiewitsch2016finite})
Let $M_1, M_2\in\ydh$ be both finite dimensional and assume 
$c_{M_1,M_2}c_{M_2,M_1}=\mathrm{id}_{M_2\otimes M_1}$.
Then $\mathfrak{B}(M_1\oplus M_2)\simeq  \mathfrak{B}(M_1)
\otimes \mathfrak{B}(M_2)$ as graded vector spaces and 
$\mathrm{GKdim}\,\mathfrak{B}(M_1\oplus M_2)=
\mathrm{GKdim}\,\mathfrak{B}(M_1)+
\mathrm{GKdim}\,\mathfrak{B}(M_2)$.
\end{lemma}

\begin{proposition}(\cite[Radford biproduct]{radford1985structure})
$H$ is a Hopf algebra. Let $A\in{}_H^H\mathcal{YD}$ be a braided 
Hopf algebra. Then $A\# H$ is a 
Hopf algebra.
\begin{eqnarray}
(a\# h)(a^\prime\# h^\prime)&=&\sum
a(h_{(1)}\cdot a^\prime)\# h_{(2)}h^\prime,\quad
a,a^\prime\in A, h,h^\prime\in H
\\
\Delta(a\# h)&=&\sum\left[a_{(1)}\#
(a_{(2)})_{(-1)}h_{(1)}\right]\otimes
\left[(a_{(2)})_{(0)}\# h_{(2)}\right]
\\
S(a\# h)&=&\sum
\left(1\# S_H(h)
S_H(a_{(-1)})\right)
\left(S_A(a_{(0)})\# 1\right)
\end{eqnarray}
\end{proposition}
The map $\iota: H \to A\#H$ given by $\iota(h) = 1\#h$ for all
$h\in H$ is an injective Hopf algebra map, and the map $\pi: A\#H \to H$
given by $\pi(a\#h) = \varepsilon_{A}(a)h$ for all $a\in A$, $h\in H$
is a surjective Hopf algebra map such that $\pi \circ \iota = \mathrm{id}_{H} $.
Moreover, it holds that $A= (A\#H)^{\mathrm{co}\, \pi}$.

Conversely, let $B$ be a Hopf algebra with bijective antipode and
$\pi: B\to H$ a Hopf algebra epimorphism admitting
a Hopf algebra section $\iota: H\to B$ such that $\pi\circ\iota =\mathrm{id}_{H}$.
Then $A=B^{\mathrm{co}\,\pi}$ is a braided Hopf algebra in $\ydh$ and $B\simeq A\# H$
as Hopf algebras.

\subsection{GK-dimension}
Let $A$ be a finitely generated algebra over a field $\mathbb{K}$, and assume $a_1, \cdots , a_m$ generate
$A$. Set $V$ to be the span of $a_1, \cdots, a_m$, and denote $V^n$ the span of all monomials in the
$a_i$'s of length $n$. As $a_i$'s generate $A$ one has $A =\bigcup_{k=0}^\infty A_k $ where $A_k =\mathbb{K}+V +V^2+\cdots+V^k$.
The function $d_V (n) = \dim A_n$ is the
growth function of $A$.
The \emph{Gelfand-Kirillov dimension} of a $\mathbb{K}$-algebra $A$ is 
$\mathrm{GKdim}~A=\varlimsup\mathrm{log}_n~d_V(n)$.
 $\mathrm{GKdim}~A$ does not depend on the choice of $V$.
Suppose that $\mathrm{GKdim} ~A< \infty$. We say that a finite-dimensional subspace
$V\subseteq A$ is  \emph{$GK$-deterministic} if 
\begin{equation}
\mathrm{GKdim} ~A = \lim_{n\rightarrow \infty} \log_n \dim
\sum_{0\leq j\leq n} V^n.
\end{equation}
Clearly, if $V$ is a $GK$-deterministic subspace of $A$, then any finite-dimensional
subspace of $A$ containing $V$ is $GK$-deterministic.
Let $A$ and $B$ be two algebras. Then
\begin{equation}
\mathrm{GKdim}(A \otimes B) \leq \mathrm{GKdim}~ A + \mathrm{GKdim} ~B,
\end{equation}
but the equality does not hold in general. For instance, it does hold when A
or B has a GK-deterministic subspace, see \cite[Proposition 3.11]{MR1721834}. The Gelfand-Kirillov dimension is a  useful tool in Ring theory and Hopf algebraic theories. We shall not discuss in detail its importance but we refer the reader to \cite{MR1721834}
as a basic reference and  \cite{MR3490761,MR3061686,MR2661247,andruskiewitsch2016finite} for additional informations related with Hopf algebras.

%
%
\section{Simple Yetter-Drinfel'd modules of $H_8$}\label{YDMH8}
Recall that the neither commutative nor cocommutative  semisimple  $8$-dimensional
Hopf algebra $H_8$ in  \cite{MR1357764} is constructed as an extension
of $\mathbb{K}[C_2 \times C_2]$ by $\mathbb{K}[C_2]$. A basis for $H_8$ is given by $\{1, x, y, xy = yx, z, xz,yz, xyz\}$ with the relations $$x^{2} =y^{2}= 1,\quad z^{2} = \frac{1}{2}(1+x+y-xy), \quad xy = yx,\quad zx=yz,\quad zy=xz.$$
The coalgebra structure and the antipode are defined by
$$
\Delta(x) = x \otimes x, \quad \Delta(y) = y \otimes y, \quad \varepsilon(x) = \varepsilon(y)=1, \quad S(x) = x, \quad S(y) = y,$$
$$
\Delta(z) =\frac{1}{2} (1 \otimes 1 + 1 \otimes x+ y \otimes 1- y \otimes x )(z\otimes z), \quad \varepsilon(z) = 1,
\quad S(z) = z.
$$
The automorphism group of $H_8$ is the Klein four-group \cite{MR2879228}. These automorphisms are given in Table \ref{autotableforH8}, which are going to be used in Corollary \ref{Isomorphism:B(V)H_8}.
\begin{center}
\begin{table}
\begin{tabular}{|c||c|c|c|c|}
\hline
& $1$ & $x$& $y$& $z$\\
\hline
$\tau_1=\mathrm{id} $& $1$ & $x$& $y$& $z$\\
$\tau_2$& $1$ & $x$& $y$ & $xyz$\\
$\tau_3$& $1$ & $y$& $x$& $\frac{1}{2}\left(z+xz+yz-xyz\right)$\\
$\tau_4$& $1$ & $y$& $x$& $\frac{1}{2}\left(-z+xz+yz+xyz\right)$\\
\hline
\end{tabular}
\vspace{3ex}
\caption{Automorphisms of $H_8$}
\label{autotableforH8}
\end{table}
\end{center}

Denote a set of central orthogonal idempotents of $H_8$ as 
$$e_1=\frac{1}{8}(1+x)(1+y)(1+z),\quad e_2=\frac{1}{8}(1+x)(1+y)(1-z),$$
$$e_3=\frac{1}{8}(1-x)(1-y)(1+\mi\, z),\quad e_4=\frac{1}{8}(1-x)(1-y)(1-\mi\,z),$$
$$e_5=\frac{1-xy}{2}, \quad e_je_k=\delta_{jk},\quad j,k=1,\cdots, 5;  \mi=\sqrt{-1}.$$
And denote idempotents $e_5^\prime=\frac{1}{4}(1-xy)(1+z), 
e_5^{\prime\prime}=\frac{1}{4}(1-xy)(1-z)$, then 
\begin{eqnarray*}
H_8 &=& H_8e_1\oplus H_8e_2\oplus H_8e_3\oplus H_8e_4\oplus H_8e_5\\
&=& H_8e_1\oplus H_8e_2\oplus H_8e_3\oplus H_8e_4\oplus(H_8e_5^{\prime}+ H_8e_5^{\prime\prime}),
\end{eqnarray*}
where $H_8e_5^{\prime}\simeq H_8e_5^{\prime\prime}$ as left $H_8$-module, via $e_5^{\prime}\mapsto xe_5^{\prime\prime}, xe_5^{\prime}\mapsto e_5^{\prime\prime}$.
\begin{definition}
Denote $V_1(b)\coloneqq\mathbb{K}\{v|x\cdot v=b^2v, y\cdot v=b^2v, z\cdot v=bv, b\in \{\pm 1, \pm \mi\}\}$, where $v$ is a vector. Let $V_2\simeq H_8e_5^{\prime}$ as left $H_8$-module, the actions of the generators are given by $$x\mapsto \begin{pmatrix}
    0  &    1\\
    1  &  0
\end{pmatrix}, \quad 
y\mapsto \begin{pmatrix}
    0  &    -1\\
    -1  &  0
\end{pmatrix}, \quad 
z\mapsto \begin{pmatrix}
    1 &    0\\
    0  &  -1
\end{pmatrix}.$$
\end{definition}
\begin{proposition}
All simple left modules of $H_8$ are classified by
$V_1(b), V_2, b\in \{\pm 1, \pm \mi\}.$
\end{proposition}
\begin{remark}
Thanks to the referee's reminder, the result was also obtained in \cite{MR1987013}
under a different basis.
\end{remark}
In the remaining part of the article, $V_1(b)$ and $V_2$ always mean a simple left $H_8$-module.

\begin{lemma} (\cite[Proposition 2]{radford2003oriented})
\label{constructionLRYD}
Let $H$ is a bialgebra over  field $\mathbb{K}$ and suppose $S$ is the antipode of $H$.
\begin{enumerate}
\item\label{que:LRYDM} If $L\in{}_{H}\mathcal{M}$, then $L\otimes H\in{}_H\mathcal{YD}^{H}$, the module and comodule actions are given by  
\begin{equation*}
h\cdot(\ell\otimes a)=h_{(2)}\cdot\ell\otimes h_{(3)}aS^{-1}(h_{(1)}),~~\rho(\ell\otimes h)=(\ell\otimes h_{(1)})\otimes h_{(2)},
\forall h, a\in H, \ell\in L.
\end{equation*}
Let $M\in {}_H \mathcal{YD}^H$ .
\item 
Suppose that $L \in {}_H\mathcal{M}$ and $p: M\longrightarrow L$ is a map of left $H$-modules. Then the linear
map $f :M \longrightarrow L \otimes H$ defined by 
$f (m) = p(m_{(0)})\otimes m_{(1)}$ for all $m\in M$ is a map of
Yetter-Drinfel'd $H$-modules, where $L \otimes H$ has the structure described in part \eqref{que:LRYDM}.
Furthermore $\ker f$ is the largest Yetter-Drinfel'd $H$-submodule, indeed the largest
subcomodule, contained in $\ker p$.
\item $M$ is isomorphic to a Yetter-Drinfel'd 
submodule of some $L\otimes H$ described in above.
\end{enumerate}
\end{lemma}

Similarly according to Radford's method, any simple left Yetter-Drinfel'd module over $H_8$ could be constructed by the submodule of  tensor product of a left  module $V$ of $H_8$ and
$H_8$ itself, where the  module and comodule structures are given by :
\begin{eqnarray}
h\cdot (\ell\boxtimes g)=(h_{(2)}\cdot \ell)\boxtimes h_{(1)}gS(h_{(3)}),\label{eq:action}\\
\rho(\ell\boxtimes h)=h_{(1)}\otimes (\ell\boxtimes h_{(2)}) , \forall h, g,\in H_8, \ell\in V.
\label{eq:coaction}
\end{eqnarray}
Here we use $\boxtimes$ instead of $\otimes$ to avoid confusion by using too many symbals of the tensor product. We are going to construct all simple left Yetter-Drinfel'd modules over $H_8$ in this way. Keeping in mind that  $H_8$ is semisimple,  it's possibly being done. In fact, it is much easier than making use of the fact that ${}_{H}^{H}\mathcal{YD}\simeq 
{}_{\mathcal{D}\left(H_8^{cop}\right)}\mathcal{M}$. The following is a list of  useful formulae for looking for simple objects of ${}_{H_8}^{H_8}\mathcal{YD}$.
\begin{lemma}
\begin{eqnarray}
(\mathrm{id}^{\otimes 2}\otimes S)\Delta^{(2)}(z)=
\frac{1}{4}[&
(1+y)z\otimes z\otimes z(1+x)+(1-y)z\otimes xz\otimes z(1+x)+\\
+&(1+y)z\otimes yz\otimes z(1-x)+(y-1)z\otimes xyz\otimes z(1-x)]
\end{eqnarray}
\begin{eqnarray}
z_{(2)}\otimes z_{(1)}?S(z_{(3)})=
\frac{1}{4}[z\otimes (1+y)z?z(1+x)+
xz\otimes(1-y)z?z(1+x)+\\
+yz\otimes(1+y)z?z(1-x)+xyz\otimes(1-y)z?z(x-1)].
\end{eqnarray}
\begin{eqnarray}
z_{(2)}\otimes z_{(1)}S(z_{(3)})=
\frac{1}{4}[z\otimes (1+x)(1+y)+
xz\otimes(1+x)(1-y)+\label{eq:1}\\
+yz\otimes(1-x)(1+y)+xyz\otimes(1-x)(1-y)],\\
z_{(2)}\otimes z_{(1)}xS(z_{(3)})=
\frac{1}{4}[z\otimes (1+x)(1+y)+
xz\otimes(1+x)(y-1)+\label{eq:x}\\
+yz\otimes(1-x)(1+y)+xyz\otimes(x-1)(1-y)],\\
z_{(2)}\otimes z_{(1)}yS(z_{(3)})=
\frac{1}{4}[z\otimes (1+x)(1+y)+
xz\otimes(1+x)(1-y)+\label{eq:y}\\
+yz\otimes(x-1)(1+y)+xyz\otimes(x-1)(1-y)],\\
z_{(2)}\otimes z_{(1)}xyS(z_{(3)})=
\frac{1}{4}[z\otimes (1+x)(1+y)+
xz\otimes(1+x)(y-1)+\label{eq:xy}\\
+yz\otimes(x-1)(1+y)+xyz\otimes(1-x)(1-y)],\\
z_{(2)}\otimes z_{(1)}zS(z_{(3)})=
\frac{1}{2}[z\otimes (1+y)z+xyz\otimes x(y-1)z],\label{eq:z}\\
z_{(2)}\otimes z_{(1)}xzS(z_{(3)})=
\frac{1}{2}[z\otimes (1+y)z+xyz\otimes x(1-y)z],\label{eq:xz}\\
z_{(2)}\otimes z_{(1)}yzS(z_{(3)})=
\frac{1}{2}[z\otimes x(1+y)z+xyz\otimes (y-1)z],\label{eq:yz}\\
z_{(2)}\otimes z_{(1)}xyzS(z_{(3)})=
\frac{1}{2}[z\otimes x(1+y)z+xyz\otimes (1-y)z].\label{eq:xyz}
\end{eqnarray}
\end{lemma}
\begin{definition}
Define $M\langle b,g\rangle\coloneqq\mathbb{K}\{v\boxtimes g|v\in V_1(b)\}$, where $b\in\{\pm 1, \pm \mi\}$ and $g\in\{1,x,y,xy\}$.
\end{definition}
\begin{lemma}
There are eight pairwise non-isomorphic one dimensional Yetter-Drinfel'd modules over $H_8$ as $M\langle b,g\rangle$ with 
$(b,g)\in\{(\pm1,1),(\pm1,xy),(\pm\mi,x),(\pm\mi,y)\}$. The actions and coactions are given by 
\begin{eqnarray}
x\cdot (v\boxtimes g)=b^2(v\boxtimes g),\quad 
y\cdot (v\boxtimes g)=b^2(v\boxtimes g), \quad
z\cdot (v\boxtimes g)=b(v\boxtimes g),\\
\rho(v\boxtimes g)=g\otimes (v\boxtimes g),\quad v\boxtimes g\in 
M\langle b,g\rangle, \quad v\in V_1(b).
\end{eqnarray}
\end{lemma}
\begin{proof}
Let $v\in V_1(b)$, then 
\begin{eqnarray}
z\cdot (v\boxtimes 1)\xlongequal{\eqref{eq:1}}
\frac{bv}{4}\boxtimes [1+x+b^2(1-x)][1+y+b^2(1-y)],\label{eq:1dimXgrouplike1}\\
z\cdot (v\boxtimes xy)\xlongequal{\eqref{eq:xy}}
\frac{bv}{4}\boxtimes [1+x+b^2(x-1)][1+y+b^2(y-1)],\\
z\cdot (v\boxtimes x)\xlongequal{\eqref{eq:x}}
\frac{bv}{4}\boxtimes [1+x+b^2(1-x)][1+y+b^2(y-1)],\\
z\cdot (v\boxtimes y)\xlongequal{\eqref{eq:y}}
\frac{bv}{4}\boxtimes [1+x+b^2(x-1)][1+y+b^2(1-y)].\label{eq:1dimXgrouplike4}
\end{eqnarray}
so 
\begin{eqnarray}
z\cdot (v\boxtimes 1)=bv\boxtimes 1,\quad
z\cdot (v\boxtimes xy)=bv\boxtimes xy,\text{when} ~ b=\pm 1;\\
z\cdot (v\boxtimes x)=bv\boxtimes x,\quad
z\cdot (v\boxtimes y)=bv\boxtimes y,\text{when} ~ b=\pm \mi.
\end{eqnarray}
Now it's easy to see that $M\langle b,g\rangle$ defined in above is a one-dimensional Yetter-Drinfel'd modules by Radford's method  and the eight one-dimensional Yetter-Drinfel'd modules are pairwise non-isomorphic by observations on their actions and coactions.
\end{proof}
\begin{definition}
Let $(g_1, g_2)\in \{(1,y), (x,1),(xy,x),(y,xy)\}$ and denote three vector spaces as 
$$M\langle(1,xy)\rangle\coloneqq\mathbb{K}\{v\boxtimes 1, v\boxtimes xy|v\in V_1(\mi)\},$$ 
$$M\langle(x,y)\rangle\coloneqq\mathbb{K}\{v\boxtimes x, v\boxtimes y|v\in V_1(1)\},$$
 $$M\langle(g_1,g_2)\rangle\coloneqq\mathbb{K}\{(v_1+v_2)\boxtimes g_1, (v_1-v_2)\boxtimes g_2|v_1,v_2\in V_2\}.$$
\end{definition}
\begin{lemma}\label{YDM2dim1}
There are six pairwise non-isomorphic two-dimensional simple Yetter-Drinfel'd modules over $H_8$ as below, where the action and coaction are given by formulae (\ref{eq:action}) and (\ref{eq:coaction}).
\begin{enumerate}
\item $M\langle(1,xy)\rangle$, the actions of generators on $(v\boxtimes 1, v\boxtimes xy)$ are given by 
$$x\mapsto \begin{pmatrix}
    -1  &    0\\
    0  &  -1
\end{pmatrix}, \quad 
y\mapsto \begin{pmatrix}
    -1  &    0\\
    0  &  -1
\end{pmatrix}, \quad 
z\mapsto \begin{pmatrix}
    0 &    \mi\\
    \mi  &  0
\end{pmatrix}.$$
\item $M\langle(x,y)\rangle$, the actions of generators on $(v\boxtimes x, v\boxtimes y)$ are given by $$x\mapsto \begin{pmatrix}
    1  &    0\\
    0  &  1
\end{pmatrix}, \quad 
y\mapsto \begin{pmatrix}
    1  &    0\\
    0  &  1
\end{pmatrix}, \quad 
z\mapsto \begin{pmatrix}
    0 &    1\\
    1  &  0
\end{pmatrix}.$$
\item $M\langle(g_1,g_2)\rangle$, where $(g_1,g_2)\in \{(1,y), (x,1),(xy,x),(y,xy)\}$. the actions of generators on the row vector $((v_1+v_2)\boxtimes g_1, (v_1-v_2)\boxtimes g_2)$ are given by 
$$x\mapsto \begin{pmatrix}
    1  &    0\\
    0  &  -1
\end{pmatrix}, \quad 
y\mapsto \begin{pmatrix}
    -1  &    0\\
    0  &  1
\end{pmatrix}, \quad 
z\mapsto \begin{pmatrix}
    0 &    1\\
    1  &  0
\end{pmatrix}.$$
\end{enumerate}
\end{lemma}
\begin{proof}
For the coactions are easy to see, we can focus on their structures as  left $H_8$-modules. Part (1) and (2) of the lemma can be checked by formulae (\ref{eq:1dimXgrouplike1}) to (\ref{eq:1dimXgrouplike4}).
Let $v_1, v_2\in V_2$, then 
\begin{eqnarray}
z\cdot (v_1\boxtimes 1)\xlongequal{\eqref{eq:1}}\frac{1}{2}[v_1\boxtimes (x+y)+v_2\boxtimes (x-y)],\\
z\cdot (v_2\boxtimes 1)\xlongequal{\eqref{eq:1}}\frac{1}{2}[v_1\boxtimes (-x+y)+v_2\boxtimes (-x-y)],\\
z\cdot (v_1\boxtimes xy)\xlongequal{\eqref{eq:xy}}\frac{1}{2}[v_1\boxtimes (x+y)+v_2\boxtimes (y-x)],\\
z\cdot (v_2\boxtimes xy)\xlongequal{\eqref{eq:xy}}\frac{1}{2}[v_1\boxtimes (x-y)+v_2\boxtimes (-x-y)],\\
z\cdot (v_1\boxtimes y)\xlongequal{\eqref{eq:y}}\frac{1}{2}[v_1\boxtimes (1+xy)+v_2\boxtimes (1-xy)],\\
z\cdot (v_2\boxtimes y)\xlongequal{\eqref{eq:y}}\frac{1}{2}[v_1\boxtimes (-1+xy)+v_2\boxtimes (-1-xy)],\\
z\cdot (v_1\boxtimes x)\xlongequal{\eqref{eq:x}}\frac{1}{2}[v_1\boxtimes (1+xy)+v_2\boxtimes (-1+xy)],\\
z\cdot (v_2\boxtimes x)\xlongequal{\eqref{eq:x}}\frac{1}{2}[v_1\boxtimes (1-xy)+v_2\boxtimes (-1-xy)].
\end{eqnarray}
So we have 
$$z\cdot [(v_1+v_2)\boxtimes 1]=(v_1-v_2)\boxtimes y,\quad 
z\cdot [(v_1-v_2)\boxtimes y]=(v_1+v_2)\boxtimes 1,$$
$$z\cdot [(v_1+v_2)\boxtimes x]=(v_1-v_2)\boxtimes 1,\quad z\cdot [(v_1-v_2)\boxtimes 1]=(v_1+v_2)\boxtimes x,$$
$$z\cdot [(v_1+v_2)\boxtimes xy]=(v_1-v_2)\boxtimes x,\quad 
z\cdot [(v_1-v_2)\boxtimes x]=(v_1+v_2)\boxtimes xy,$$
$$z\cdot [(v_1+v_2)\boxtimes y]=(v_1-v_2)\boxtimes xy,\quad  
z\cdot [(v_1-v_2)\boxtimes xy]=(v_1+v_2)\boxtimes y.$$
Part (3) is immediate to check. The six two-dimensional Yetter-Drinfel'd modules are pairwise non-isomorphic since they are 
pairwise non-isomorphic as comodules.
\end{proof}

\begin{lemma}
\label{YDM2dim2}
Let $b_1,b_2\in \{\pm 1\} $ and $v\in V_1(b_2)$ and denote 
\begin{equation}
w^{b_1,b_2}_1\triangleq v\boxtimes (1+\mi b_1 y)z,\quad w^{b_1,b_2}_2\triangleq v\boxtimes x(1-\mi b_1 y)z.
\end{equation}
Then $W^{b_1,b_2}=\mathbb{K}w^{b_1,b_2}_1+\mathbb{K}w^{b_1,b_2}_2$ is  a family of 4 pairwise non-isomorphic two dimensional simple  Yetter-Drinfel'd modules over $H_8$ with the actions of generators on the row vector $(w^{b_1,b_2}_1, w^{b_1,b_2}_2)$ and coactions  given by 
$$x\mapsto \begin{pmatrix}
    0  &     -\mi b_1\\
  \mi b_1 &  0
\end{pmatrix}, \quad 
y\mapsto \begin{pmatrix}
   0  &     -\mi b_1\\
    \mi b_1 &  0
\end{pmatrix}, \quad 
z\mapsto \begin{pmatrix}
    \frac{(1-\mi b_1)b_2}{2} &    \frac{(1-\mi b_1)b_2}{2}\\
    \frac{-(1-\mi b_1)b_2}{2}  &  \frac{(1-\mi b_1)b_2}{2}
\end{pmatrix},$$
$$\rho\left(w^{b_1,b_2}_1\right)=\frac{(1+y)z}{2}\otimes w^{b_1,b_2}_1
+\frac{(1-y)z}{2}\otimes w^{b_1,b_2}_2,$$
$$\rho\left(w^{b_1,b_2}_2\right)=\frac{x(1+y)z}{2}\otimes w^{b_1,b_2}_2
+\frac{x(1-y)z}{2}\otimes w^{b_1,b_2}_1.$$
\end{lemma}
\begin{proof}
 It's straightforward by the definition of Yetter-drinfel'd module. When $b_2\neq b_2^\prime$,   $W^{b_1,b_2}\nsimeq W^{b_1,b_2^\prime}$ since we  will see that their braidings  are different in Proposition \ref{NAlgdim1}. As explained in the following remark, $W^{b_1,b_2}$ has another basis $\{p_1,p_2\}$ with $p_1\in V_1(b_2)$ and $p_2\in V_1(-b_1b_2\mi)$. So $W^{b_1,b_2}\nsimeq W^{b_1^\prime,b_2}$ if $b_1\neq b_1^\prime$.
 \end{proof}
 \begin{remark}
 \begin{enumerate}
 \item
Let
$M=\mathbb{K}\{v\boxtimes z, v\boxtimes xz, v\boxtimes yz, v\boxtimes xyz| v\in V_1(b)\}$, $b\in\{\pm 1\}$.
$z$ acts on elements of $M$ as
\begin{align*}
z\cdot (v\boxtimes z)&\xlongequal{\eqref{eq:z}}
\frac{bv}{2}\boxtimes (1-x+y+xy)z,
&z\cdot (v\boxtimes xz)\xlongequal{\eqref{eq:xz}}
\frac{bv}{2}\boxtimes (1+x+y-xy)z,\\
z\cdot (v\boxtimes yz)&\xlongequal{\eqref{eq:yz}}
\frac{bv}{2}\boxtimes (-1+x+y+xy)z,
&z\cdot (v\boxtimes xyz)\xlongequal{\eqref{eq:xyz}}
\frac{bv}{2}\boxtimes (1+x-y+xy)z.
\end{align*}
Then $M\simeq W^{1,b}\oplus W^{-1,b}$ as Yetter-Drinfel'd modules over $H_8$. 
\item Let $f_{jk}\triangleq\frac{1}{4}[1+(-1)^jx][1+(-1)^ky]$, $j,k=0,1$. Denote 
$$p_1=w_1^{b_1, b_2}+\mi b_1 w_2^{b_1,b_2},\quad 
p_2=w_1^{b_1, b_2}-\mi b_1 w_2^{b_1,b_2},$$
then $W^{b_1,b_2}=\mathbb{K}p_1+\mathbb{K}p_2$  with the actions of generators on the row vector $(p_1, p_2)$ and coactions  given by 
$$x\mapsto \begin{pmatrix}
    1  &     0\\
  0 &  -1
\end{pmatrix}, \quad 
y\mapsto \begin{pmatrix}
    1  &     0\\
  0 &  -1
\end{pmatrix}, \quad 
z\mapsto \begin{pmatrix}
    b_2 &    0\\
    0  &  -\mi b_1b_2
\end{pmatrix},$$
$$\rho(p_1)=\left[f_{00}-
\mi b_1f_{11}\right]z\otimes p_1+\left[f_{10}+
\mi b_1f_{01}\right]z\otimes p_2,$$
$$\rho(p_2)=\left[f_{00}+
\mi b_1f_{11}\right]z\otimes p_2+
\left[f_{10}-
\mi b_1f_{01}\right]z\otimes p_1.$$
\end{enumerate}
\end{remark}

According to \cite[Remark 2.14]{MR1357764}, $H_8$ is presented by generators $x,y,w$, where the expressions containing $z$ are replaced by
\begin{eqnarray}
w = \left(f_{00}+\sqrt{\mi} f_{10} +\frac{1}{\sqrt{\mi}}f_{01}+\mi f_{11}\right)z, \quad
 w^2=1, \\
 wx = yw,\quad
S(w)=\left(\frac{1+\mi}{2}x+\frac{1-\mi}{2}y\right)w,\\
\Delta(w)=\left(\frac{1}{2}(1+xy)\otimes 1+\frac{1+\mi}{4}(1-xy)\otimes x+\frac{1-\mi}{4}(1-xy)\otimes y\right)(w\otimes w).
\end{eqnarray}
Let $a+1=\pm \sqrt{2}$, we define
\begin{eqnarray*}
w_1^{(1)}\triangleq(v_1+\mi av_2)\boxtimes \frac{\mi}{2}\left[(x+y)+\sqrt{\mi}(x-y)\right]w
+(av_1-\mi v_2)\boxtimes \frac{1}{2}\left[(x+y)-\sqrt{\mi}(x-y)\right]w,\\
w_2^{(1)}\triangleq(v_1+\mi av_2)\boxtimes \frac{\mi}{2}\left[(1+xy)+\sqrt{\mi}(1-xy)\right]w
-(av_1-\mi v_2)\boxtimes \frac{1}{2}\left[(1+xy)-\sqrt{\mi}(1-xy)\right]w,\\
w_1^{(2)}\triangleq(v_1-\mi av_2)\boxtimes \frac{\mi}{2}\left[(x+y)+\sqrt{\mi}(x-y)\right]w
+(av_1+\mi v_2)\boxtimes \frac{1}{2}\left[(x+y)-\sqrt{\mi}(x-y)\right]w,\\
w_2^{(2)}\triangleq(v_1-\mi av_2)\boxtimes \frac{\mi}{2}\left[(1+xy)+\sqrt{\mi}(1-xy)\right]w
-(av_1+\mi v_2)\boxtimes \frac{1}{2}\left[(1+xy)-\sqrt{\mi}(1-xy)\right]w.
\end{eqnarray*}
\begin{lemma}
\label{YDM2dim3}
Let $a+1=\pm \sqrt{2}$, there are  4 pairwise non-isomorphic 
simple Yetter-Drinfel'd modules $W_1^a$ and $W_2^a$ over $H_8$ as following
\begin{enumerate}
\item Let $W_1^a=\mathbb{K}w_1^{(1)}\oplus \mathbb{K}w_1^{(1)}$,  then $W_1^a$ is a two dimensional simple  Yetter-Drinfel'd module over $H_8$ with actions  given by 
\[ \left\{
\begin{array}{rl}
&x\cdot w_1^{(1)}=-w_1^{(1)}  \\
&y\cdot w_1^{(1)}=w_1^{(1)} \\
&z\cdot w_1^{(1)}=\frac{1}{2}(1-\mi)(a+1)w_2^{(1)}\\
&w\cdot w_1^{(1)}=\frac{1}{2\sqrt{\mi}}(1-\mi)(a+1)w_2^{(1)} 
\end{array}
\right.\quad
\left\{
\begin{array}{rl}
&x\cdot w_2^{(1)}=w_2^{(1)}  \\
&y\cdot w_2^{(1)}=-w_2^{(1)} \\
&z\cdot w_2^{(1)}=\frac{1}{2}(1+\mi)(a+1)w_1^{(1)}\\
&w\cdot w_2^{(1)}=\frac{\sqrt{\mi}}{2}(1+\mi)(a+1)w_1^{(1)}
\end{array}
\right.
\]
and coactions given by
\begin{align*}
\rho\left(w_1^{(1)}\right)&=\frac{1}{2}(x+y)w\otimes w_1^{(1)}+\frac{\sqrt{\mi}}{2} (x-y)w
\otimes w_2^{(1)},\\
\rho\left(w_2^{(1)}\right)&=\frac{1}{2}(1+xy)w\otimes w_2^{(1)}+\frac{\sqrt{\mi}}{2} (1-xy)w
\otimes w_1^{(1)}.
\end{align*}
\item Let $W_2^a=\mathbb{K}w_1^{(2)}\oplus \mathbb{K}w_1^{(2)}$, then $W_2^a$ is a two dimensional simple  Yetter-Drinfel'd module over $H_8$ with actions  given by 
\[ \left\{
\begin{array}{rl}
&x\cdot w_1^{(2)}=w_1^{(2)}  \\
&y\cdot w_1^{(2)}=-w_1^{(2)} \\
&z\cdot w_1^{(2)}=\frac{1}{2}(1-\mi)(a+1)w_2^{(2)}\\
&w\cdot w_1^{(2)}=\frac{\sqrt{\mi}}{2}(1-\mi)(a+1)w_2^{(2)}
\end{array}
\right.\quad 
\left\{
\begin{array}{rl}
&x\cdot w_2^{(2)}=-w_2^{(2)}  \\
&y\cdot w_2^{(2)}=w_2^{(2)} \\
&z\cdot w_2^{(2)}=\frac{1}{2}(1+\mi)(a+1)w_1^{(2)}\\
&w\cdot w_2^{(2)}=\frac{1}{2\sqrt{\mi}}(1+\mi)(a+1)w_1^{(2)}
\end{array}
\right.
\]
and coactions given by 
\begin{align*}
\rho\left(w_1^{(2)}\right)&=\frac{1}{2}(x+y)w\otimes w_1^{(2)}+\frac{\sqrt{\mi}}{2} (x-y)w
\otimes w_2^{(2)} ,\\
\rho\left(w_2^{(2)}\right)&=\frac{1}{2}(1+xy)w\otimes w_2^{(2)}+\frac{\sqrt{\mi}}{2} (1-xy)w
\otimes w_1^{(2)}.
\end{align*}
\end{enumerate}
\end{lemma}
\begin{proof}
It's straightforward to check by the definition of Yetter-drinfel'd module. Actually 
$M\simeq \bigoplus_{a+1=\pm \sqrt{2}}\left(W^a_1 \oplus W^a_2\right)$ as Yetter-Drinfel'd modules over $H_8$, where
$M=\mathbb{K}\{v_j\boxtimes z, v_j\boxtimes xz, v_j\boxtimes yz, v_j\boxtimes xyz|v_j\in V_2, j=1,2\}$. 

Since $\sqrt{\mi}=\cos\frac{\pi}{4}+\mi \sin\frac{\pi}{4}$, 
$\frac{1}{2}\sqrt{2}\sqrt{\mi}(1-\mi)=1$. Denote $a+1=b \sqrt{2}$, $b=\pm 1$, $p_1^{(1)}=\sqrt{\mi}w_1^{(1)}+w_2^{(1)}$, $p_2^{(1)}=-\sqrt{\mi}w_1^{(1)}+w_2^{(1)}$, then $W^a_1=\mathbb{K}p_1^{(1)}+\mathbb{K}p_2^{(1)}$ with actions on the row vector  
$\left(p_1^{(1)}, p_2^{(1)}\right)$ given by 
$$x\mapsto \begin{pmatrix}
    0  &    1\\
    1  &  0
\end{pmatrix}, \quad 
y\mapsto \begin{pmatrix}
    0  &    -1\\
    -1  &  0
\end{pmatrix}, \quad 
z\mapsto \begin{pmatrix}
    b &    0\\
    0  &  -b
\end{pmatrix}.$$
Let $p_1^{(2)}=w_1^{(2)}+\frac{1}{\sqrt{\mi}}w_2^{(2)}$, $p_2^{(2)}=w_1^{(2)}-\frac{1}{\sqrt{\mi}}w_2^{(2)}$, then $W^a_2=\mathbb{K}p_1^{(2)}+\mathbb{K}p_2^{(2)}$ with actions on the row vector  
$\left(p_1^{(2)}, p_2^{(2)}\right)$ given by 
$$x\mapsto \begin{pmatrix}
    0  &    1\\
    1  &  0
\end{pmatrix}, \quad 
y\mapsto \begin{pmatrix}
    0  &    -1\\
    -1  &  0
\end{pmatrix}, \quad 
z\mapsto \begin{pmatrix}
    b &    0\\
    0  &  -b
\end{pmatrix}.$$
Now we can observe that $W_1^{-1+\sqrt{2}}$ is isomorphic to  $W_1^{-1-\sqrt{2}}$ as modules (or comodules) under suitably chosen base, but they are not isomorphic as modules and comodules at the same time. So 
$W_1^{-1+\sqrt{2}}\nsimeq W_1^{-1-\sqrt{2}}$ as Yetter-Drinfel'd modules. For the same reason, we have $W_2^{-1+\sqrt{2}}\nsimeq W_2^{-1-\sqrt{2}}$ and $W_1^a\nsimeq W_2^a$.
\end{proof}

Obviously, any  module in Lemma \ref{YDM2dim1}
is not isomorphic to any one of  modules in Lemma \ref{YDM2dim2} and \ref{YDM2dim3} as comodules. As $H_8$-modules, $W^{b_1,b_2}\simeq V_1(b_2)\oplus V_1(-b_1b_2\mi)$, and $W_1^a\simeq W_2^a\simeq V_2$. So Yetter-drinfel'd modules in Lemma \ref{YDM2dim1}, \ref{YDM2dim2} and \ref{YDM2dim3} are pairwise non-isomorphic. Keeping in mind that $H_8$ is semisimple, now we are arriving at
\begin{theorem} All the simple  Yetter-Drinfel'd modules over $H_8$ are classified by
\begin{itemize}
\item 8 pairwise non-isomorphic simple Yetter-drinfel'd modules of one-dimension:
$$M\langle b,g\rangle,\quad (b,g)\in\left\{(\pm 1,1),(\pm 1,xy), (\pm \mi,x), (\pm \mi,y)\right\}.$$
\item 14 pairwise non-isomorphic simple Yetter-drinfel'd modules of two-dimension: 
$$M\langle(1,xy)\rangle, ~M\langle(x,y)\rangle, 
M\langle(g_1,g_2)\rangle,
 W^{b_1,b_2}, W_1^a, W_2^a, $$
where $(g_1,g_2)\in \{(1,y), (x,1),(xy,x),(y,xy)\}, b_1, b_2\in\{\pm 1\}, a+1=\pm\sqrt{2}.$
\end{itemize}
\end{theorem}
\begin{remark}
Jun Hu  and Yinhuo Zhang  investigated $\mathcal{D}(H)$-modules in \cite{MR2336009} and \cite{MR2352888} by using Radford's construction \cite{radford2003oriented}, especially they constructed all simple modules of $\mathcal{D}(H_8)$ under a different basis comparing with ours. \end{remark}

\section{Nichols algebras in ${}_{H_8}^{H_8}\mathcal{YD}$}
\label{section:NicholsAlgebras}
In this section,  we try to determine all the finite-dimensional Nichols algebras generated by  Yetter-Drinfel'd modules
over $H_8$. As a byproduct, we calculate Gelfand-Kirillov dimensions for some  Nichols algebras.

We begin by studying the Nichols algebras of  simple Yetter-Drinfel'd modules.
\begin{proposition}\label{NAlgdim1}
Given a simple Yetter-Drinfel'd module $M$ over $H_8$, $\dim\mathfrak{B}(M)$ ($\mathrm{Gkdim}\,\mathfrak{B}(M)$ for some cases) is presented in Table \ref{dimNAlgSimpleYDM_H8}. Especially, 
 \begin{center}
\begin{table}
\begin{tabular}{|c||c|c|c|}
\hline
 $M\in {}_{H_8}^{H_8}\mathcal{YD}$
 & condition & $\dim\mathfrak{B}(M)$ & $\mathrm{GKdim}\,\mathfrak{B}(M)$\\
\hline
 \multirow{2}{*}
{$M\langle b,g\rangle$}& $(b,g)\in\{(\pm 1, 1),(\pm 1, xy)\}$ & $\infty$& $1$\\\cline{2-4}
 &$(b,g)\in\{(\pm\mi, x),(\pm\mi, y)\}$ & $2$& $0$ \\\hline
$M\langle(1,xy)\rangle$
&  & $\infty$& $2$\\\hline
$M\langle(x,y)\rangle$
&  & $\infty$& $2$\\\hline
 \multirow{2}{*}{$M\langle(g_1,g_2)\rangle$}& $(g_1,g_2)\in\{(1,y),(x,1)\}$ & $\infty$& $\infty$\\\cline{2-4}
& $(g_1,g_2)\in\{(xy,x),(y,xy)\}$ & $8$& $0$\\\hline
 \multirow{2}{*}{$W^{b_1,b_2}$}
& $b_1=\pm 1$, $b_2=-1$ & $8$& $0$\\\cline{2-4}
& $b_1=\pm 1$, $b_2=1$ & $\infty$& $\infty$\\\hline
$W_1^a$, $W_2^a$& $a+1=\pm\sqrt{2}$ & $\infty$& \\
\hline
\end{tabular}
\vspace{3ex}
\caption{Nichols algebras of simple Yetter-Drinfel'd modules over $H_8$}
\label{dimNAlgSimpleYDM_H8}
\end{table}
\end{center}
\begin{enumerate}
\item \label{NicholsAlg:M(b,g)} 
$ \mathfrak{B}\left(M\langle b,g\rangle\right) = \left\{
\begin{array}{rl}
\mathbb{K}[p], & \text{if } (b,g)\in\{(\pm 1,1),(\pm 1,xy)\},\\
\mathbb{K}[p]/(p^{2})= \bigwedge \mathbb{K}p, &\text{if } (b,g)\in\{(\pm \mi,x), (\pm \mi,y)\}.\\
\end{array}
\right.$
\item  \label{NicholsAlg:A_2} The both braidings of $M\langle(g_1,g_2)\rangle$ for $(g_1,g_2)\in \{(xy,x),(y,xy)\}$ and $W^{b_1,-1}$ for $b_1=\pm 1$ are   Cartan type $A_2$, so their corresponding  Nichols algebras 
are isomorphic to an algebra which is generated 
by $p_1$, $p_2$ satisfying relations 
$p_1p_2p_1p_2+p_2p_1p_2p_1=0$, $p_1^2=p_2^2=0.$
\end{enumerate}
\end{proposition}
\begin{proof}
\begin{itemize}\renewcommand{\labelitemi}{$\diamond$}
\item 
 Because 
$ c(p\otimes p) =g\cdot p\otimes p= \left\{
\begin{array}{rl}
p\otimes p, & \text{if } (b,g)\in\{(\pm 1,1),(\pm 1,xy)\}\\
-p\otimes p, &\text{if } (b,g)\in\{(\pm \mi,x), (\pm \mi,y)\}\\
\end{array}
\right.
$
under the assumption that $M\langle b,g\rangle=\mathbb{K}p$, 
the part \eqref{NicholsAlg:M(b,g)}  is obvious.
\item As for the part \eqref{NicholsAlg:A_2}, we only give a proof for the case $W^{b_1,-1}$ for $b_1=\pm 1$. Let 
$p_1=w_1^{b_1,b_2}+\mi b_1w_2^{b_1,b_2}$ and 
$p_2=w_1^{b_1,b_2}-\mi b_1w_2^{b_1,b_2}$, then 
the braiding of $W^{b_1,b_2}$ is given by
\begin{eqnarray*}
c(p_1\otimes p_1)&=&b_2 p_1\otimes p_1, \quad
c(p_2\otimes p_2)=b_2 p_2\otimes p_2, \\
c(p_1\otimes p_2)&=&-b_2 p_2\otimes p_1, ~~
c(p_2\otimes p_1)=b_2 p_1\otimes p_2. 
\end{eqnarray*}
When $b_2=1$, $\mathrm{GKdim}\,\mathfrak{B}\left(W^{b_1,1}\right)=\infty$ according to \cite[Lemma 2.8]{andruskiewitsch2016finite}.  
When  $b_2=-1$, the braiding is type $A_2$. As discussed in \cite{MR2136919}, the Nichols algebra 
$\mathfrak{B}\left(W^{b_1,-1}\right)$ is generated 
by $p_1$, $p_2$ with relations 
$p_1p_2p_1p_2+p_2p_1p_2p_1=0$, $p_1^2=p_2^2=0.$
So $\mathrm{dim}\left(\mathfrak{B}\left(W^{b_1,-1}\right)\right)=8$. 
\item
 Let $p_1=v\boxtimes 1, p_2=v\boxtimes xy\in M\langle(1,xy)\rangle$, then 
 $c(p_j\otimes p_k)=p_k\otimes p_j$, where $j,k=1,2$. 
 If we view $M\langle(1,xy)\rangle=\mathbb{K}p_1\oplus \mathbb{K}p_2$ as braided vector spaces, then 
$\mathrm{GKdim}~\mathfrak{B}\left(M\langle(1,xy)
\rangle\right)= \mathrm{GKdim}\,\mathfrak{B}(\mathbb{K}p_1)+
\mathrm{GKdim}\,\mathfrak{B}(\mathbb{K}p_2)=2$ by Lemma \ref{TensorNicholsAlg}.
Similarly, 
$\mathrm{GKdim}\,\mathfrak{B}\left(M\langle(x,y)\rangle\right)=2$.
\item 
Let $p_1=(v_1+v_2)\boxtimes 1$, 
$p_2=(v_1-v_2)\boxtimes y\in 
M\langle(1,y)\rangle$. The braiding is given by
\begin{eqnarray*}
c(p_1\otimes p_1)=p_1\otimes p_1,\quad
c(p_1\otimes p_2)=p_2\otimes p_1,\\
c(p_2\otimes p_1)=-p_1\otimes p_2,\quad
c(p_2\otimes p_2)=p_2\otimes p_2.
\end{eqnarray*}
By \cite[Lemma 2.8]{andruskiewitsch2016finite}, 
$\mathrm{GKdim}\mathfrak{B}\left(M\langle(1,y)\rangle\right)=\infty$. For the same reason,  we obtain $\mathrm{GKdim}\mathfrak{B}\left(M\langle(x,1)\rangle\right)=\infty$.
\item Let $\theta=\frac{1}{2}(\mi-1)(a+1)$, then 
\begin{eqnarray*}
c\left(w_1^{(1)}\otimes w_1^{(1)}\right)=
-\theta w_2^{(1)}\otimes w_2^{(1)},\quad
c\left(w_1^{(1)}\otimes w_2^{(1)}\right)=
\theta w_1^{(1)}\otimes w_2^{(1)},\\
c\left(w_2^{(1)}\otimes w_1^{(1)}\right)=
-\theta w_2^{(1)}\otimes w_1^{(1)},\quad
c\left(w_2^{(1)}\otimes w_2^{(1)}\right)=
\theta w_1^{(1)}\otimes w_1^{(1)},\\
c\left(\mi w_1^{(1)}\otimes w_1^{(1)}+w_2^{(1)}\otimes w_2^{(1)}\right)
=-\mi\theta \left(\mi w_1^{(1)}\otimes w_1^{(1)}+w_2^{(1)}\otimes w_2^{(1)}\right), \\
c\left(-\mi w_1^{(1)}\otimes w_1^{(1)}+w_2^{(1)}\otimes w_2^{(1)}\right)
=\mi\theta \left(-\mi w_1^{(1)}\otimes w_1^{(1)}+w_2^{(1)}\otimes w_2^{(1)}\right), 
\end{eqnarray*}
By induction, 
\begin{eqnarray*}
\mathfrak{S}_{2n-1,1}\left(\left(w_1^{(1)}\otimes w_2^{(1)}\right)^{\otimes n}\right)=\frac{(1+\theta)[1-(-\theta^2)^n]}{1+\theta^2}\left(\left(w_1^{(1)}\otimes w_2^{(1)}\right)^{\otimes n}\right),\\
\mathfrak{S}_{2n,1}\left(\left(w_1^{(1)}\otimes w_2^{(1)}\right)^{\otimes n}\otimes w_1^{(1)}\right)=\frac{1-\theta+(-1)^n\theta^{2n+1}(1+\theta)}{1+\theta^2}\left(\left(w_1^{(1)}\otimes w_2^{(1)}\right)^{\otimes n}\otimes w_1^{(1)}\right).
\end{eqnarray*} 
It means that $\left(w_1^{(1)}\otimes w_2^{(1)}\right)^{\otimes n}$ is an eigenvector of $\mathfrak{S}_{2n-1}$ and 
$\left(w_1^{(1)}\otimes w_2^{(1)}\right)^{\otimes n}\otimes w_1^{(1)}$ is an eigenvector 
of $\mathfrak{S}_{2n}$ both with nonzero eigenvalue.
 So $\dim\mathfrak{B}\left(W_1^a\right)=\infty$. And $\dim\mathfrak{B}\left(W_2^a\right)=\infty$ is similar to prove.
\end{itemize}
\end{proof}

\begin{proposition}\label{NicholsAlg:TensorOfTwoSimpleObjects}
\begin{enumerate}
\item  \label{NicholsAlg:Tensor1}$\mathfrak{B}\left[M\langle b,g\rangle\oplus 
M\left<b^\prime,g^\prime\right>\right]\simeq 
\mathfrak{B}\left(M\langle b,g\rangle\right)
\otimes 
\mathfrak{B}\left( 
M\left<b^\prime,g^\prime\right>\right)$ for 
 $(b,g)$, $(b^\prime,g^\prime)$ $\in$ $\{(\pm 1,1)$, $(\pm 1,xy)$,  $(\pm \mi,x)$, $(\pm \mi,y)\}$.
\item  When $(b,g)\in\{(\pm1,1),(\pm1,xy)\}$, then 
\begin{align*}
\mathfrak{B}\left[M\langle b,g\rangle\oplus
M\langle(1,xy)\rangle\right]
&\simeq \mathfrak{B}\left(M\langle b,g\rangle\right)\otimes 
 \mathfrak{B}\left(M\langle(1,xy)\rangle\right),\\
\mathfrak{B}\left[M\langle b,g\rangle\oplus
M\langle(x,y)\rangle\right]
&\simeq \mathfrak{B}\left(M\langle b,g\rangle\right)\otimes 
 \mathfrak{B}\left(M\langle(x,y)\rangle\right).
\end{align*}
\item  \label{NicholsAlg:Tensor3}$\mathfrak{B}\left[M\langle b,g\rangle\oplus
M\langle(g_1,g_2)\rangle\right]
\simeq \mathfrak{B}\left(M\langle b,g\rangle\right)\otimes 
 \mathfrak{B}\left(M\langle(g_1,g_2)\rangle\right)$ for the following cases
       \begin{enumerate}
       \item $(b,g)=(\pm\mi, x)$, $(g_1,g_2)=(xy,x)$;
       \item $(b,g)=(\pm\mi, y)$, $(g_1,g_2)=(y,xy)$;
       \item $(b,g)=(\pm1,1)$, $(g_1,g_2)\in$  
        $\{(xy, x),(y, xy)\}$.
       \end{enumerate}
\item \label{NicholsAlg:Tensor4}$\mathfrak{B}\left[M\langle b,g\rangle\oplus
W^{b_1,-1}\right]
\simeq \mathfrak{B}\left(M\langle b,g\rangle\right)\otimes 
 \mathfrak{B}\left(W^{b_1,-1}\right)$ for the following cases
       \begin{enumerate}
       \item $(b,g)\in\{(1, 1), (1,xy)\}$, $b_1=\pm 1$;
       \item $(b,g)\in\{(\mi, x), (\mi,y)\}$, $b_1= 1$;
       \item $(b,g)\in\{(-\mi, x), (-\mi,y)\}$, $b_1= -1$.
       \end{enumerate}
\item $\mathfrak{B}\left[M\langle(xy,x)\rangle
\oplus M\langle(y,xy)\rangle\right]
\simeq \mathfrak{B}\left(M\langle(xy,x)\rangle\right)
\otimes 
\mathfrak{B}\left(M\langle(y,xy)\rangle\right)$.
\item \label{NicholsAlg:Tensor6}$\mathfrak{B}\left(W^{1,-1}\oplus W^{-1,-1}\right)
\simeq \mathfrak{B}\left(W^{1,-1}\right)
\otimes 
\mathfrak{B}\left(W^{-1,-1}\right)$.
\item \label{NicholsAlg:Tensor7}$\mathrm{GKdim}\,\mathfrak{B}\left[M\langle b,g\rangle\oplus
M\langle(g_1,g_2)\rangle\right] =\infty$ for $(b,g)=(\pm\mi, x)$, $(g_1,g_2)=(y,xy)$ or $(b,g)=(\pm\mi, y)$, $(g_1,g_2)=(xy,x)$.
\item \label{NicholsAlg:Tensor8}$\mathrm{GKdim}\,\mathfrak{B}\left[M\langle b,g\rangle\oplus
W^{b_1,-1}\right] =\infty$ for $(b,g)\in\{(\mi, x), (\mi,y)\}$, $b_1= -1$  or $(b,g)\in\{(-\mi, x), (-\mi,y)\}$, $b_1= 1$.
\item \label{NicholsAlg:Tensor9}$\mathrm{dim}~\mathfrak{B}\left(\left(M\langle(g_1,g_2)\rangle\right)^{\oplus 2}\right)=\infty$  for
 $(g_1,g_2)\in \{(xy,x),(y,xy)\}$.
 \item \label{NicholsAlg:Tensor10}$\dim \mathfrak{B}\left(W^{b_1,-1}\oplus W^{b_1,-1}
\right)=\infty$ for $b_1=\pm 1$.
\end{enumerate}
\end{proposition}
\begin{remark}
According to the above two proposition, we calculate some Nichols algebras over direct sum of two simple objects of ${}_{H_8}^{H_8}\mathcal{YD}$ in Table \ref{dimNAlgTwoSimpleYDM_H8}.
\end{remark}
 \begin{center}
\begin{table}
\newcommand{\tabincell}[2]{\begin{tabular}{@{}#1@{}}#2\end{tabular}}
\begin{tabular}{|c|c|c|c|}
\hline
 $M\in {}_{H_8}^{H_8}\mathcal{YD}$
 & condition & $\dim\mathfrak{B}(M)$ & $\mathrm{GKdim}\mathfrak{B}(M)$\\
\hline
 \multirow{5}{*}
{\tabincell{c}{$M\langle b_1,g_1\rangle$\\
$\oplus M\langle b_2,g_2\rangle$}}
&\tabincell{c}{$(b_1,g_1)$, $(b_2,g_2)\in$
\\ $\{(\pm 1, 1),(\pm 1, xy)\}$ }& $\infty$& $2$\\\cline{2-4}
&\tabincell{c}{$(b_1,g_1)\in$$\{(\pm 1, 1),(\pm 1, xy)\}$
\\$(b_2,g_2)\in$  $\{(\pm\mi, x),(\pm\mi, y)\}$}
& $\infty$& $1$\\\cline{2-4}
 &\tabincell{c}{$(b_1,g_1)$, $(b_2,g_2)\in$
\\$\{(\pm\mi, x),(\pm\mi, y)\}$ }& $4$& $0$ \\\hline
\tabincell{c}{$M\langle b,g\rangle$\\
 $\oplus M\langle(1,xy)\rangle$}
& $(b,g)\in$ $\{(\pm 1, 1),(\pm 1, xy)\}$ & $\infty$& $3$\\\hline
\tabincell{c}{$M\langle b,g\rangle$ \\
$\oplus M\langle(x,y)\rangle$}
&$(b,g)\in$ $\{(\pm 1, 1),(\pm 1, xy)\}$& $\infty$& $3$\\\hline
 \multirow{3}{*}{\tabincell{c}{$M\langle(g_1,g_2)\rangle$\\
 $\oplus M\langle(g_1^\prime,g_2^\prime)\rangle$}}& 
 \tabincell{c}{$(g_1,g_2)=(g_1^\prime,g_2^\prime)=$\\
$(xy,x)$ or $(y,xy)$ }& $\infty$& $$\\\cline{2-4}
&\tabincell{c}{ $(g_1,g_2)=(xy,x)$, \\$(g_1^\prime,g_2^\prime)
=(y,xy)$} & $64$& $0$\\\hline
 \multirow{2}{*}{$W^{b_1,-1}\oplus W^{b_1^\prime,-1}$}
& $b_1=b_1^\prime=\pm 1$ & $\infty$& $$\\\cline{2-4}
& $b_1=1$, $b_1^\prime=-1$ & $64$& $0$\\\hline
$M\langle(g_1,g_2)\rangle\oplus W^{b_1,-1}$& 
\tabincell{c}{ $(g_1,g_2)=(xy,x)$ \\or $(y,xy)$, $b_1=\pm 1$ }& ? & \\
\hline
 \multirow{6}{*}
{\tabincell{c}{$M\langle b,g\rangle$\\
$\oplus M\langle(g_1,g_2)\rangle$}}
&\tabincell{c}{$(b,g)=(\pm\mi, x)$, $(g_1,g_2)=(xy,x)$ }& $16$& $0$\\\cline{2-4}
&\tabincell{c}{$(b,g)=(\pm\mi, x)$, $(g_1,g_2)=(y,xy)$ }& $\infty$& $$\\\cline{2-4}
&\tabincell{c}{$(b,g)=(\pm\mi, y)$, $(g_1,g_2)=(xy,x)$ }& $\infty$& $$\\\cline{2-4}
&\tabincell{c}{$(b,g)=(\pm\mi, y)$, $(g_1,g_2)=(y,xy)$ }& $16$& $0$\\\cline{2-4}
&\tabincell{c}{$(b,g)\in$$\{(\pm 1, 1)\}$
\\$(g_1,g_2)\in$  $\{(xy, x),(y, xy)\}$}
& $\infty$& $1$\\\hline
 \multirow{6}{*}
{\tabincell{c}{$M\langle b,g\rangle$
$\oplus W^{b_1,-1}$}}
&\tabincell{c}{$(b,g)\in\{(1, 1), (1,xy)\}$, $b_1=\pm 1$ }& $\infty$& $1$\\\cline{2-4}
&\tabincell{c}{$(b,g)\in\{(\mi, x), (\mi,y)\}$, $b_1= 1$ }& $16$& $0$\\\cline{2-4}
&\tabincell{c}{$(b,g)\in\{(\mi, x), (\mi,y)\}$, $b_1= -1$  }& $\infty$& $$\\\cline{2-4}
&\tabincell{c}{$(b,g)\in\{(-\mi, x), (-\mi,y)\}$, $b_1= -1$ }& $16$& $0$\\\cline{2-4}
&\tabincell{c}{$(b,g)\in\{(-\mi, x), (-\mi,y)\}$, $b_1= 1$  }& $\infty$& $$\\\hline
\end{tabular}
\vspace{3ex}
\caption{Nichols algebras over direct sum of two simple objects in ${}_{H_8}^{H_8}\mathcal{YD}$}
\label{dimNAlgTwoSimpleYDM_H8}
\end{table}
\end{center}

\begin{proof}
The part \eqref{NicholsAlg:Tensor1}-\eqref{NicholsAlg:Tensor6} are direct results of  Lemma  \ref{TensorNicholsAlg}. We only prove some cases as a byproduct in the following.
\begin{itemize}\renewcommand{\labelitemi}{$\diamond$}
\item 
Let $p_1=(v_1+v_2)\boxtimes g_1$, 
$p_2=(v_1-v_2)\boxtimes g_2\in M\langle(g_1,g_2)\rangle$, where $(g_1,g_2)\in\{(xy,x),(y,xy)\}$. Let $p=v\boxtimes g
\in M\langle b,g\rangle$, then
\begin{eqnarray*}
c(p\otimes p_1) &=& \left\{
\begin{array}{rl}
-p_1\otimes p, & \text{if }g\in\{y,xy\}\\
p_1\otimes p, &\text{if } g\in\{1,x\},\\
\end{array}\right.
\quad
c(p\otimes p_2) = \left\{
\begin{array}{rl}
-p_2\otimes p, & \text{if }g\in\{x,xy\}\\
p_2\otimes p, &\text{if } g\in\{1,y\},\\
\end{array}\right.
\\
c(p_1\otimes p) &=& \left\{
\begin{array}{rl}
b^2p\otimes p_1, & \text{if } g_1=y\\
p\otimes p_1, &\text{if } g_1=xy,\\
\end{array}\right.
\quad~~
c(p_2\otimes p) = \left\{
\begin{array}{rl}
b^2p\otimes p_2, & \text{if }g_2=x\\
p\otimes p_2, &\text{if } g_2=xy.\\
\end{array}\right.
\end{eqnarray*}
\begin{itemize}\renewcommand{\labelitemii}{$\circ$}
\item When $(g_1,g_2)=(y,xy)$ and $(b,g)=(\pm\mi,x)$, then
          \begin{align*}
          c(p\otimes p_1)&=p_1\otimes p_1, 
       &c(p\otimes p_2)&=-p_2\otimes p,\\
          c(p_1\otimes p)&=-p\otimes p_1,
       &c(p_2\otimes p)&=p\otimes p_2.   
          \end{align*}
          The generalized Dynkin diagram is given by Figure \ref{figureEightA1A2}. According to \cite{heckenberger2009classification}, 
          $\dim \mathfrak{B}[M\langle\pm\mi,x\rangle
          \oplus M\langle(y,xy)\rangle]=\infty$.
          \begin{figure}[h!]
$$
\xy 
(0,0)*\cir<4pt>{}="E1", 
(30,0)*\cir<4pt>{}="E2",
(15,25.98)*\cir<4pt>{}="E3",
(-4,0)*+{-1},
(34,0)*+{-1},
(14,2)*+{-1},
(11,25.98)*+{-1},
(4.5,13)*+{-1},
(26,13)*+{-1},
\ar @{-}"E1";"E2"
\ar @{-}"E2";"E3"
\ar @{-}"E3";"E1"
\endxy
$$
\caption{}
\label{figureEightA1A2}
\end{figure}
\item When $(g_1,g_2)=(xy,x)$ and $(b,g)=(\pm\mi,y)$,  
the generalized Dynkin diagram associated to the braiding is given by Figure \ref{figureEightA1A2}. According to \cite{heckenberger2009classification}, 
          $\dim \mathfrak{B}[M\langle\pm\mi,y\rangle
          \oplus M\langle(xy,x)\rangle]=\infty$. We finish  the part \eqref{NicholsAlg:Tensor7}.
\item As for cases listed in the part \eqref{NicholsAlg:Tensor6},  $\mathfrak{B}\left[M\langle b,g\rangle\oplus
M\langle(g_1,g_2)\rangle\right]
\simeq \mathfrak{B}\left(M\langle b,g\rangle\right)\otimes 
 \mathfrak{B}\left(M\langle(g_1,g_2)\rangle\right)$ by Lemma \ref{TensorNicholsAlg}.
\end{itemize}
\item Let $p=v\boxtimes g\in M\langle b,g\rangle$, where $(b,g)\in\{(\pm 1,1),(\pm 1,xy), (\pm \mi,x), (\pm \mi,y)\}$, then 
\begin{eqnarray*}
c(p\otimes w_1^{b_1,-1}) &=& \left\{
\begin{array}{rl}
w_1^{b_1,-1}\otimes p, & \text{if }~
 (b,g)\in\{(\pm 1,1),(\pm 1,xy)\}\\
\mi b_1 w_2^{b_1,-1}\otimes p, &\text{if }~ 
(b,g)\in\{(\pm\mi,x),(\pm \mi,y)\},\\
\end{array}\right.
\\
c(p\otimes w_2^{b_1,-1}) &=& \left\{
\begin{array}{rl}
w_2^{b_1,-1}\otimes p, & \text{if }~
 (b,g)\in\{(\pm 1,1),(\pm 1,xy)\}\\
-\mi b_1 w_1^{b_1,-1}\otimes p, &\text{if }~
 (b,g)\in\{(\pm\mi,x),(\pm \mi,y)\},\\
\end{array}\right.
\\
c(w_1^{b_1,-1}\otimes p) &=& \left\{
\begin{array}{rl}
bp\otimes w_1^{b_1,-1}, & \text{if } ~
 (b,g)\in\{(\pm 1,1),(\pm 1,xy)\}\\
bp\otimes w_2^{b_1,-1}, &\text{if } ~
 (b,g)\in\{(\pm\mi,x),(\pm \mi,y)\},\\
\end{array}\right.
\\
c(w_2^{b_1,-1}\otimes p) &=& \left\{
\begin{array}{rl}
bp\otimes w_2^{b_1,-1}, & \text{if }~
 (b,g)\in\{(\pm 1,1),(\pm 1,xy)\}\\
-bp\otimes w_1^{b_1,-1}, &\text{if }~
 (b,g)\in\{(\pm\mi,x),(\pm \mi,y)\}.\\
\end{array}\right.
\end{eqnarray*}
\begin{itemize}\renewcommand{\labelitemii}{$\circ$}
\item In case $(b,g)\in\{(1,1),(1,xy)\}$,  according to Lemma \ref{TensorNicholsAlg}, we have 
$$\mathfrak{B}\left(M\langle b,g\rangle\oplus W^{b_1, -1}\right)\simeq \mathfrak{B}\left(M\langle b,g\rangle\right)\otimes \mathfrak{B}\left( W^{b_1, -1}\right).$$
\item In case $(b,g)\in\{(\pm\mi,x),(\pm \mi,y)\}$, if $\mi b_1b=-1$,  according to Lemma \ref{TensorNicholsAlg}, we have 
$\mathfrak{B}\left(M\langle b,g\rangle\oplus W^{b_1, -1}\right)\simeq \mathfrak{B}\left(M\langle b,g\rangle\right)\otimes \mathfrak{B}\left( W^{b_1, -1}\right)$.
If $\mi b_1b=1$, the generalized Dynkin diagram associated to the braiding of $M\langle b,g\rangle$ $\oplus$ $W^{b_1, -1}$ is 
given by Figure \ref{figureEightA1A2}. Now we finish the part 
\eqref{NicholsAlg:Tensor4} and \eqref{NicholsAlg:Tensor8}.
\end{itemize}
\item As for $(g_1,g_2)\in \{(xy,x),(y,xy)\}$,  $\mathrm{dim}~\mathfrak{B}\left(\left(M\langle(g_1,g_2)\rangle\right)^{\oplus 2}\right)=\infty$ by \cite{heckenberger2009classification},  since the generalized Dynkin
diagram associated to the braiding is given by Figure
\ref{figureEightM23}.
\begin{figure}[h!]
$$
\xy 
(0,0)*\cir<4pt>{}="E1", 
(20,0)*\cir<4pt>{}="E2",
(20,20)*\cir<4pt>{}="E3",
(0,20)*\cir<4pt>{}="E4",
(-4,0)*+{-1},
(-3,10)*+{-1},
(-4,20)*+{-1},
(24,20)*+{-1},
(24,0)*+{-1},
(23,10)*+{-1},
(10,22)*+{-1},
(10,2)*+{-1},
\ar @{-}"E1";"E2"
\ar @{-}"E2";"E3"
\ar @{-}"E4";"E1"
\ar @{-}"E4";"E3"
\endxy
$$
\caption{}
\label{figureEightM23}
\end{figure}
\item As for $W^{b_1,-1}\oplus W^{b_1^\prime,-1}$
with $b_1$ and $b_1^\prime$ in $\{\pm 1\}$. Let 
$p_1=w_1^{b_1,-1}+\mi b_1^\prime w_2^{b_1,-1}$,  
$p_2=w_1^{b_1,-1}-\mi b_1^\prime w_2^{b_1,-1}$, 
$p_1^\prime=w_1^{b_1^\prime,-1}+\mi b_1^\prime w_2^{b_1^\prime,-1}$ and 
$p_2^\prime=w_1^{b_1^\prime,-1}-\mi b_1^\prime w_2^{b_1^\prime,-1}$, then
\begin{align*}
c(p_1\otimes p_1^\prime)&=- p_1^\prime\otimes p_1, 
&c(p_2\otimes p_2^\prime)&=- p_2^\prime\otimes p_2, \\
c(p_1\otimes p_2^\prime)&= p_2^\prime\otimes p_1,
&c(p_2\otimes p_1^\prime)&=- p_1^\prime\otimes p_2. 
\end{align*}
When $b_1=b_1^\prime$, the generalized Dynkin
diagram associated to the braiding is given by Figure
\ref{figureEightM23}. By \cite{heckenberger2009classification},  
$\dim \mathfrak{B}\left(W^{b_1,-1}\oplus W^{b_1,-1}
\right)=\infty$. This finish \eqref{NicholsAlg:Tensor10}.

When $b_1=-b_1^\prime$, then $p_2=w_1^{b_1,-1}+\mi b_1 w_2^{b_1,-1}$,  
$p_1=w_1^{b_1,-1}-\mi b_1 w_2^{b_1,-1}$, 
$p_2^\prime=w_1^{b_1^\prime,-1}+\mi b_1 w_2^{b_1^\prime,-1}$,  
$p_1^\prime=w_1^{b_1^\prime,-1}-\mi b_1 w_2^{b_1^\prime,-1}$, and 
\begin{align*}
c(p_2^\prime\otimes p_2)&=- p_2\otimes p_2^\prime, 
&c(p_1^\prime\otimes p_1)&=- p_1\otimes p_1^\prime, \\
c(p_2^\prime\otimes p_1)&= p_1\otimes p_2^\prime,
&c(p_1^\prime\otimes p_2)&=- p_2\otimes p_1^\prime. 
\end{align*}
By Lemma \ref{TensorNicholsAlg}, we have
$\mathfrak{B}\left(W^{b_1,-1}\oplus W^{-b_1,-1}\right)
\simeq \mathfrak{B}\left(W^{b_1,-1}\right)
\otimes 
\mathfrak{B}\left(W^{-b_1,-1}\right)$. This finish \eqref{NicholsAlg:Tensor6}.
\end{itemize}
\end{proof}

\begin{conjecture}\label{conjecture}
$\dim\mathfrak{B}\left(M\langle(xy,x)\rangle
\oplus W^{b_1,-1}\right)=\infty
=\dim\mathfrak{B}\left(M\langle(y,xy)\rangle
\oplus W^{b_1,-1}\right)$ hold for $b_1=\pm 1$.
\end{conjecture}
\begin{remark}
Let $p_1=(v_1+v_2)\boxtimes xy$,  $p_2=(v_1-v_2)\boxtimes x$
$\in M\langle(xy,x)\rangle$, then 
\begin{align*}
c\left(p_1\otimes w_1^{b_1,-1}\right)&=w_1^{b_1,-1}\otimes p_1, 
&c\left(p_1\otimes w_2^{b_1,-1}\right)&=w_2^{b_1,-1}\otimes p_1,\\
c\left(p_2\otimes w_1^{b_1,-1}\right)&=\mi b_1 w_2^{b_1,-1}\otimes p_2, 
&c\left(p_2\otimes w_2^{b_1,-1}\right)&=-\mi b_1w_1^{b_1,-1}\otimes p_2,\\
c\left(w_1^{b_1,-1}\otimes p_1\right)&=p_2\otimes w_1^{b_1,-1}, 
&c\left(w_1^{b_1,-1}\otimes p_2\right)&=p_1\otimes w_2^{b_1,-1},\\
c\left(w_2^{b_1,-1}\otimes p_1\right)&=-p_2\otimes w_2^{b_1,-1}, 
&c\left(w_2^{b_1,-1}\otimes p_2\right)&=p_1\otimes w_1^{b_1,-1}.
\end{align*}
Let 
$p_1^\prime=w_1^{b_1,-1}+\mi b_1w_2^{b_1,-1}$ and 
$p_2^\prime=w_1^{b_1,-1}-\mi b_1w_2^{b_1,-1}$, then
\begin{align*}
c(p_1\otimes p_1^\prime)&=p_1^\prime\otimes p_1,
&c(p_1\otimes p_2^\prime)&=p_2^\prime\otimes p_1,\\
c(p_2\otimes p_1^\prime)&=p_1^\prime\otimes p_2,
&c(p_2\otimes p_2^\prime)&=-p_2^\prime\otimes p_2,\\
c(p_1^\prime\otimes p_1)&=p_2\otimes p_2^\prime,
&c(p_1^\prime\otimes p_2)&=\mi b_1 p_1\otimes p_2^\prime,\\
c(p_2^\prime\otimes p_1)&=p_2\otimes p_1^\prime,
&c(p_2^\prime\otimes p_2)&=-\mi b_1 p_1\otimes p_1^\prime.
\end{align*}
\end{remark}

 \subsection*{Proof of Theorem \ref{NicholsAlg:maintheorem}}
Firstly, we recall the truth that  for any submodule $
 M_1\subset M_2\in{}_{H}^{H}\mathcal{YD}$, $\mathfrak{B}(M_1)
 \subset \mathfrak{B}(M_2)$, $\dim \mathfrak{B}(M_2)=\infty$ if $\dim \mathfrak{B}(M_1)=\infty$. So the only possible $M\in{}_{H_8}^{H_8}\mathcal{YD}$ such that $\dim\mathfrak{B}(M)< \infty$ is in the list of Theorem \ref{NicholsAlg:maintheorem} according to 
Table  \ref{dimNAlgSimpleYDM_H8},  Table \ref{dimNAlgTwoSimpleYDM_H8}, 
   Proposition \ref{NAlgdim1} and 
   Proposition \ref{NicholsAlg:TensorOfTwoSimpleObjects}, under the assumption that Conjecture \ref{conjecture} is true. Now we only need to check that Nichols algebras $ \mathfrak{B}(M)$ for $M$ listed in Theorem \ref{NicholsAlg:maintheorem} is finite dimensional. In fact,  $\Omega_1(n_1,n_2,n_3,n_4)$ is of Cartan type $\underbrace{A_1\times \cdots \times A_1}_{n_1+n_2+n_3+n_4}$;
  $\Omega_k(n_1,n_2)$ for $k=2,3,4,5$ is of Cartan type 
  $\underbrace{A_1\times \cdots \times A_1}_{n_1+n_2}\times\, A_2$; 
 $\Omega_k$ for $k=6,7$ is of Cartan type $A_2\times A_2$.

\section{Hopf algebras over $H_8$}
\label{section:HopfAlgebras}
In this section, according to lifting method, we determine  finite-dimensional Hopf algebra $H$ with coradical $H_8$ such that its infinitesimal braiding is isomorphic to a Yetter-Drinfel'd module $M$ listed in Theorem \ref{NicholsAlg:maintheorem}. We begin by proving $H$ is generated by elements of degree one in Theorem \ref{generatedByDegreeOne}. That is, $\mathrm{gr}\, H\simeq \mathfrak{B}(M)\#H_8$.

\begin{theorem}\label{generatedByDegreeOne}
Let $H$ be a finite-dimensional Hopf algebra over $H_8$ such that its infinitesimal braiding is isomorphic to a Yetter-Drinfel'd module  over $H_8$ which is in the list of Theorem \ref{NicholsAlg:maintheorem}. Then the diagram of $H$ is a Nichols algebra, and consequently $H$ is generated by the elements of degree one with respect to the coradical filtration.
\end{theorem}
\begin{proof}
Since $\mathrm{gr}\, H \simeq  R \#H_8$, with $R=\bigoplus_{n\geq 0} R(n)$ the diagram of $H$, we need
to prove that $R$ is a Nichols algebra.
Let $\mathcal{J}=\bigoplus_{n\geq 0} R(n)^*$ be the graded dual of $R$, then $\mathcal{J}$ is a graded Hopf algebra in ${}_{H_8}
^{H_8}\mathcal{YD}$ with $\mathcal{J}(0)=\mathbb{K}1$. 
According to \cite[Lemma 5.5]{MR1780094},   
$R(1)= \mathcal{P}(R)$ if and only if $\mathcal{J}$ is generated as an algebra by $\mathcal{J}(1)$, that is, if $\mathcal{J}$ is itself a Nichols algebra.

Considering $\mathfrak{B}(M)\in {}_{H_8}
^{H_8}\mathcal{YD}$ for  $M$  in the list of Theorem \ref{NicholsAlg:maintheorem}, since $\mathfrak{B}(M)=T(M)/\mathcal{I}$, in order to show $ \mathcal{P}(\mathcal{J})=\mathcal{J}(1)$,  it is enough to prove that the relations that generate the ideal
$\mathcal{I}$ hold in $\mathcal{J}$. This can be done by a case by case computation. We perform here three cases,
and leave  the rest to  the reader.

Suppose $M=\Omega_1(n_1,n_2,n_3,n_4)$. A direct computation shows that the elements $r$ in $\mathcal{J}$ representing
the quadratic relations are primitive and since the braiding is -flips, they
satisfy that $c(r\otimes r)=r\otimes r$. As $\dim \mathcal{J} <\infty$, it must be $r=0$ in $\mathcal{J}$ and hence there
exists a projective algebra map 
$\mathfrak{B}(M)$
$\rightarrow\mathcal{J}$, 
which implies that
$ \mathcal{P}(\mathcal{J})=\mathcal{J}(1)$.

Suppose $M=\Omega_6$, then  $M$ is generated by elements $p_1=(v_1+v_2)\boxtimes xy$, $p_2=(v_1-v_2)\boxtimes x$, $p_1^\prime=(v_1+v_2)\boxtimes y$, $p_2^\prime=(v_1-v_2)\boxtimes xy$ and the ideal defining the Nichols algebra is generated by the elements $p_1^2$, $p_2^2$, ${p_1^\prime}^2$, ${p_2^\prime}^2$, $p_1p_2p_1p_2+p_2p_1p_2p_1$, $p_1^\prime p_2^\prime p_1^\prime p_2^\prime+p_2^\prime p_1^\prime p_2^\prime p_1^\prime$, $p_1p_1^\prime+p_1^\prime p_1$, $p_1p_2^\prime+p_2^\prime p_1$, $p_2p_1^\prime-p_1^\prime p_2$, $p_2p_2^\prime+p_2^\prime p_2$.
 We can check directly that all those generators of the defining ideal of $\mathfrak{B}(M)$ are  primitive elements, or by \cite[Theorem 6]{MR2842083}. It's enough to show $c(r\otimes r)=r\otimes r$ for all 
 generators given in above for the defining ideal. Since
 $\rho(p_1)=xy\otimes p_1$,  $\rho(p_2)=x\otimes p_2$, 
 $\rho(p_1^\prime)=y\otimes p_1^\prime$, so $\rho(p_1^2)=1\otimes p_1^2$,  $\rho(p_1p_2p_1p_2+p_2p_1p_2p_1)=1\otimes (p_1p_2p_1p_2+p_2p_1p_2p_1)$, 
 $\rho(p_1p_1^\prime+p_1^\prime p_1)=x\otimes (p_1p_1^\prime+p_1^\prime p_1)$. It's easy to see $c(r\otimes r)=r\otimes r$ holds for $r=p_1^2$, $p_1p_2p_1p_2+p_2p_1p_2p_1$ and $p_1p_1^\prime+p_1^\prime p_1$. We leave the rest to the reader.

Suppose $M=\Omega_4(n_1,n_2)$,  then  
$M$ is generated by elements $p_1=w_1^{1 ,-1}+\mi  w_2^{1,-1}$, $p_2=w_1^{1 ,-1}-\mi  w_2^{1 ,-1}$, $\{X_j\}_{j=1,\cdots, n_1}$, $\{Y_k\}_{k=1,\cdots, n_2}$ with $\mathbb{K}X_j\simeq M\langle\mi,x\rangle$, $\mathbb{K}Y_k\simeq M\langle\mi,y\rangle$ and the ideal defining the Nichols algebra is generated by the elements $p_1^2$, $p_2^2$, $p_1p_2p_1p_2+p_2p_1p_2p_1$, $X_j^2$, $\{X_{j_1}X_{j_2}+
X_{j_2}X_{j_1}\}_{1\leq j_1< j_2\leq n_1}$, $Y_k^2$, $\{Y_{k_1}Y_{k_2}+Y_{k_2}Y_{k_1}\}_{1\leq k_1< k_2\leq n_2}$, $p_1Y_k-Y_k p_1$, $p_2Y_k+Y_k p_2$, $p_1X_j-X_jp_1$, $p_2X_j+X_jp_2$. We can check directly that all those generators of the defining ideal of $\mathfrak{B}(M)$ are  primitive elements, or by \cite[Theorem 6]{MR2842083}. It's enough to show $c(r\otimes r)=r\otimes r$ for all 
 generators given in above for the defining ideal. Since
 $\rho(p_1)=(f_{00}-\mi f_{11})z\otimes p_1+
(f_{10}+\mi  f_{01})z\otimes p_2$,
$\rho(p_2)=(f_{00}+\mi  f_{11})z\otimes p_2+
(f_{10}-\mi  f_{01})z\otimes p_1$,
 $\rho(X_j)=x\otimes X_j$,
 \begin{align*}
 \rho(p_1p_2p_1p_2+p_2p_1p_2p_1)
 &=[(f_{00}-\mi  f_{11})z(f_{00}+\mi  f_{11})z]^2\otimes p_1p_2p_1p_2+\\
 &\quad+[(f_{00}+\mi  f_{11})z(f_{00}-\mi  f_{11})z]^2\otimes 
 p_2p_1p_2p_1+\\
 &\quad +[(f_{10}+\mi  f_{01})z(f_{10}-\mi  f_{01})z]^2\otimes 
 p_2p_1p_2p_1+\\
 &\quad +[(f_{10}-\mi  f_{01})z(f_{10}+\mi  f_{01})z]^2\otimes 
 p_1p_2p_1p_2\\
 &=xy\otimes (p_1p_2p_1p_2+p_2p_1p_2p_1), \\
 \rho(p_1X_j-X_jp_1)=(f_{00}+\mi  &f_{11})z\otimes 
 (p_1X_j-X_jp_1)+(f_{10}-\mi  f_{01})z\otimes (p_2X_j+X_jp_2).
 \end{align*}
 Because 
 \begin{align*}
 (f_{10}-\mi f_{01})z\cdot (p_1X_j-X_jp_1)
 &=\frac{f_{10}-\mi f_{01}}{2}\cdot \left[((1+y)z\cdot p_1)(z\cdot X_j)-((1-y)z\cdot X_j)(xz\cdot p_1)\right]\\
 &=(-\mi)(f_{10}-\mi f_{01})\cdot (p_1X_j-X_jp_1)=0,\\
 (f_{00}+\mi f_{11})z\cdot (p_1X_j-X_jp_1)
 &=(-\mi)(f_{00}+\mi f_{11})\cdot (p_1X_j-X_jp_1)=p_1X_j-X_jp_1,\\
 xy\cdot (p_1p_2p_1p_2+p_2p_1p_2p_1)&=p_1p_2p_1p_2+p_2p_1p_2p_1,
 \end{align*}
 $c(r\otimes r)=r\otimes r$ holds for $r=p_1p_2p_1p_2+p_2p_1p_2p_1$ and $p_1X_j-X_jp_1$. We leave the rest to the reader.
\end{proof}

\begin{lemma}\cite[Lemma 6.1]{MR1780094}
Let $H$ be a Hopf algebra, $\psi: H\rightarrow H$ an automorphism of Hopf algebras, $V$, $W$ Yetter-Drinfel'd modules over $H$.
\begin{enumerate}
\item Let $V^\psi$ be the same space underlying $V$ but with action and coaction 
$$h\cdot_\psi v=\psi(h)\cdot v,\quad \rho^\psi (v)
=\left(\psi^{-1}\otimes \mathrm{id}\right)\rho(v), \quad h\in H, v\in V.$$ 
Then $V^\psi$ is also a Yetter-Drinfel'd module over $H$.
If $T: V\rightarrow W$ is a morphism in ${}_H^H\mathcal{YD}$, then 
$T^\psi: V^\psi\rightarrow W^\psi$ also is. Moreover, the braiding 
$c: V^\psi\otimes  W^\psi\rightarrow    W^\psi \otimes V^\psi$ 
coincides with the braiding $c: V\otimes W\rightarrow W\otimes V$.
\item If $R$ is an algebra (resp., a coalgebra, a Hopf algebra) in 
${}_H^H\mathcal{YD}$, then $R^\psi$ also is, with the same structural maps.
\item Let $R$ be a Hopf algebra in ${}_H^H\mathcal{YD}$. Then the map $\Psi: R^\psi\# H\rightarrow R\#H$ given by $\Psi(r\#h)=r\#\psi(h)$ is an isomorphism of Hopf algebras. 
\end{enumerate}
\end{lemma}
\begin{corollary}\label{Isomorphism:B(V)H_8}
\begin{enumerate}
\item $\left[M\langle b\mi, x\rangle\right]^{\tau_3}\simeq 
M\langle -b\mi, y\rangle$, $b=\pm 1$.
\item $\left[M\langle(xy, x)\rangle\right]^{\tau_3}\simeq 
M\langle(y, xy)\rangle$, $\left(W^{b_1,-1}\right)^{\tau_3}\simeq 
W^{-b_1,-1}$ with $b_1=\pm 1$.
\item $\mathfrak{B}\left(\Omega_2(n_1,n_2)\right)\#H_8\simeq 
\mathfrak{B}\left(\Omega_3(n_2,n_1)\right)\#H_8$,  
$\mathfrak{B}\left(\Omega_4(n_1,n_2)\right)\#H_8\simeq 
\mathfrak{B}\left(\Omega_5(n_2,n_1)\right)\#H_8$. 
\end{enumerate}
\end{corollary}

\begin{definition}\label{definition:HopfAlgA_1}
For $n_1, n_2, n_3, n_4\in \mathbb{N}^{\geq 0}$ with $n_1+ n_2+ n_3+ n_4\geq 1$, and a set $I_1$ of parameters $\lambda_{j,s}$, $\mu_{j,t}$,  $\zeta_{k,s}$, $\theta_{k,t}$
in $\mathbb{K}$ with $j=1,\cdots, n_1$, $k=1,\cdots, n_2$, 
$s=1,\cdots, n_3$, $t=1,\cdots, n_4$, denote by $\mathfrak{A}_1(n_1,n_2,n_3,n_4;I_1)$ the algebra generated by 
$x$, $y$, $z$,  $\{X_j\}_{j=1,\cdots, n_1}$, $\{Y_k\}_{k=1,\cdots, n_2}$, $\{p_s\}_{s=1,\cdots, n_3}$, $\{q_t\}_{t=1,\cdots, n_4}$ satisfying  the following relations:
\begin{align}
x^2=y^2=1,\quad z^2=\frac{1}{2}(1+x+y-xy), \\
xy=yx,\quad zx=yz,\quad zy=xz,\\
 xX_j=-X_jx,\quad yX_j=-X_jy,\quad zX_j=\mi X_jxz,\\
 xY_k=-Y_kx,\quad yY_k=-Y_ky,\quad zY_k=-\mi Y_kxz,\\
 xp_s=-p_sx,\quad yp_s=-p_sy,\quad zp_s=\mi p_sxz,\\
xq_t=-q_tx,\quad yq_t=-q_ty,\quad zq_t=-\mi q_txz,\\
X_j^2=0,\quad Y_k^2=0,\quad p_s^2=0,\quad q_t^2=0,\label{ideal:directsumOneDim1}\\
X_{j_1}X_{j_2}+X_{j_2}X_{j_1}=0,\quad j_1, j_2\in\{1,\cdots, n_1\},\label{ideal:directsumOneDim2}\\
Y_{k_1}Y_{k_2}+Y_{k_2}Y_{k_1}=0,\quad k_1, k_2\in\{1,\cdots, n_2\},
\label{ideal:directsumOneDim3}\\
p_{s_1}p_{s_2}+p_{s_2}p_{s_1}=0,\quad s_1,s_2\in\{1,\cdots, n_3\},
\label{ideal:directsumOneDim4}\\
q_{t_1}q_{t_2}+q_{t_2}q_{t_1}=0,\quad t_1,t_2\in\{1,\cdots, n_4\},
\label{ideal:directsumOneDim5}\\
X_jY_k+Y_kX_j=0,\quad X_jp_s+p_sX_j=\lambda_{j,s}(1-xy),
\label{ideal:directsumOneDim6}\\
X_jq_t+q_tX_j=\mu_{j,t}(1-xy),\quad
Y_kp_s+p_sY_k=\zeta_{k,s}(1-xy),
\label{ideal:directsumOneDim7}\\
Y_kq_t+q_tY_k=\theta_{k,t}(1-xy),\quad p_sq_t+q_tp_s=0.
\label{ideal:directsumOneDim8}
\end{align}
It is a Hopf algebra with its structure determined by 
\begin{align}
\Delta(X_j)=X_j\otimes 1+x\otimes X_j,\quad 
\Delta(Y_k)=Y_k\otimes 1+x\otimes Y_k,\\
\Delta(p_s)=p_s\otimes 1+y\otimes p_s,\quad 
\Delta(q_t)=q_t\otimes 1+y\otimes q_t,\\
\Delta(x)=x\otimes x,\quad \Delta(y)=y\otimes y,
\quad \Delta(z)=\frac{1}{2}[(1+y)z\otimes z+(1-y)z\otimes xz].
\end{align}
\end{definition}
\begin{remark}
\begin{enumerate}
\item In fact, 
$\mathfrak{A}_1(n_1,n_2,n_3,n_4;I_1)\simeq  \left[T\left(
\Omega_1(n_1, n_2, n_3, n_4)\right)\#H_8\right]/{\mathcal{I}(I_1)}$, 
 where ${\mathcal{I}(I_1)}$ is a Hopf ideal generated by relations\eqref{ideal:directsumOneDim1}--\eqref{ideal:directsumOneDim8}. Especially, when parameters in $I_1$ are all equal to zero, then $\mathfrak{A}_1(n_1,n_2,n_3,n_4;I_1)\simeq  \mathfrak{B}\left(
\Omega_1(n_1, n_2, n_3, n_4)\right)\#H_8$.
\item We can observe that any element of $\mathfrak{A}_1(n_1,n_2,n_3,n_4;I_1)$ can be expressed by a linear sum of 
$\{X_1^{\alpha_1}\cdots X_{n_1}^{\alpha_{n_1}}
Y_1^{\beta_1}\cdots Y_{n_2}^{\beta_{n_2}}
p_1^{\gamma_1}\cdots p_{n_3}^{\gamma_{n_3}}
q_1^{\kappa_1}\cdots q_{n_4}^{\kappa_{n_4}}x^cy^dz^e\}$ for all parameteres $\alpha_1$, $\cdots$, $\alpha_{n_1}$, $\beta_1$, $\cdots$,  $\beta_{n_2}$, $\gamma_1$, $\cdots$, $\gamma_{n_3}$, 
$\kappa_1$, $\cdots$,  $\kappa_{n_4}$, $c$, $d$, $e$ in $\{0,1\}$. So 
$\mathfrak{A}_1(n_1,n_2,n_3,n_4;I_1)$ is finite dimensional.
\end{enumerate}
\end{remark}
\begin{proposition}\label{HopfAlge:A_1}
Let $H$ be a finite-dimensional Hopf algebra with coradical $H_8$  such that its infinitesimal braiding  is isomorphic to $\Omega_1(n_1,n_2,n_3,n_4)$, then $H\simeq \mathfrak{A}_1(n_1,n_2,n_3,n_4;I_1)$.
\end{proposition} 
\begin{proof}
By Theorem \ref{generatedByDegreeOne}, we have $\mathrm{gr}\,H\simeq \mathfrak{B}(\Omega_1(n_1,n_2,n_3,n_4))\#H_8$.  We can suppose $H$ is generated by elements 
$x$, $y$, $z$ in $H$ and 
\begin{align}
X_j&=(v\boxtimes x)\#1\in M\langle\mi,x\rangle\#1,\quad j=1,\cdots, n_1,\\
Y_k&=(v\boxtimes x)\#1\in M\langle -\mi,x\rangle\#1,\quad k=1,\cdots, n_2,\\
p_s&=(v\boxtimes y)\#1\in M\langle\mi,y\rangle\#1,\quad s=1,\cdots, n_3,\\
q_t&=(v\boxtimes y)\#1\in M\langle -\mi,y\rangle\#1,\quad t=1,\cdots, n_4.
\end{align}
Then it's easy to check that formulae listed in Definition \ref{definition:HopfAlgA_1} except \eqref{ideal:directsumOneDim1}--\eqref{ideal:directsumOneDim8}
hold in $H$ from the bosonization $\mathfrak{B}[\Omega(n_1,n_2,n_3,n_4)]\#H_8$.
Let $p=(v\boxtimes g)\#1$ $\in [M\langle b,g\rangle]\#1$,  $p^\prime=(v^\prime\boxtimes g^\prime)\#1\in [M\langle b^\prime,g^\prime\rangle]\#1$, then 
$\Delta(pp^\prime+p^\prime p)=(pp^\prime+p^\prime p)
\otimes 1+gg^\prime\otimes (pp^\prime+p^\prime p)$, from which we obtain the lifting relations \eqref{ideal:directsumOneDim1}--\eqref{ideal:directsumOneDim8} are only possible for the given generators. In fact, those lifting relations \eqref{ideal:directsumOneDim1}--\eqref{ideal:directsumOneDim8}
generate a Hopf ideal. 
\end{proof}

\begin{remark}
\begin{enumerate}
\item Suppose $H_1(b)$ is a finite  dimensional Hopf algebra with coradical $H_8$ such that its infinitesimal braiding  is isomorphic to $M\langle b,x\rangle$,  
where $b=\pm \mi$, then $H_1(b)\simeq \mathfrak{B}[M\langle b,x\rangle]\#H_8$.
Denote $p=(v\boxtimes x)\# 1\in [M\langle b,x\rangle] \#1$, then  $H_1(b)$ is generated by $H_8$, $p$, and with the following relations.
$$p^2=0,\quad xp=-px,\quad yp=-py,\quad zp=bpxz,\quad \Delta(p)=p\otimes 1+x\otimes p.$$
\item Suppose $H_2(b)$ is a finite  dimensional Hopf algebra with coradical $H_8$ 
such that its infinitesimal braiding  is isomorphic to $M\langle b,y\rangle$, where $b=\pm \mi$, then $H_2(b)\simeq \mathfrak{B}[M\langle b,y\rangle]\#H_8$.
Denote $p=(v\boxtimes y)\# 1\in [M\langle b,y\rangle]\#1$, then $H_2(b)$ is generated by $H_8$, $p$, and with the following relations.
$$p^2=0,\quad xp=-px,\quad yp=-py,\quad zp=bpxz,\quad \Delta(p)=p\otimes 1+y\otimes p.$$
\item $H_1(\pm \mi)$ are exactly the two nonisomorphic nonpointed self-dual Hopf algebras of dimension
$16$ with coradical $H_8$ described by C{\u{a}}linescu, D{\u{a}}sc{\u{a}}lescu, Masuoka and Menini in
\cite{MR2037722}. In fact, $H_1(b)\simeq H_2(-b)$ by Corollary \ref{Isomorphism:B(V)H_8}.
\end{enumerate}
\end{remark}

\begin{lemma}\label{HopfAlg:M2_3(y,xy)}
Suppose $H$ is  a finite  dimensional Hopf algebra with coradical $H_8$ such that  its infinitesimal braiding  is isomorphic to $M\langle(y,xy)\rangle$,  then $H\simeq \mathfrak{B}[M\langle(y,xy)\rangle]\# H_8$.
Denote $p_1=[(v_1+v_2)\boxtimes y]\# 1\in M\langle(y,xy)\rangle]\#$ and 
$p_2=[(v_1-v_2)\boxtimes xy]\# 1\in M\langle(y,xy)\rangle]\#1$,  then $H$ is generated by $x$, $y$, $z$, $p_1$, $p_2$,  which satisfy the following relations.
\begin{eqnarray}
x^2=y^2=1,\quad z^2=\frac{1}{2}(1+x+y-xy), \quad
xy=yx,\quad zx=yz,\quad zy=xz,\\
 xp_1=p_1x,\quad yp_1=-p_1y,\quad zp_1=p_2 z,\\
  xp_2=-p_2x,\quad
  yp_2=p_2y,\quad zp_2= p_1 xz,\\
 p_1^2=0,\quad p_2^2=0,\quad 
 p_1p_2p_1p_2+p_2p_1p_2p_1=0.
 \end{eqnarray}
 Its Hopf algebra structure is determined by 
 \begin{align}
  \Delta(p_1)=p_1\otimes 1+y\otimes p_1,\quad \Delta(p_2)=p_2\otimes 1+xy\otimes p_2,\\
 \Delta(x)=x\otimes x,\quad \Delta(y)=y\otimes y,
\quad \Delta(z)=\frac{1}{2}[(1+y)z\otimes z+(1-y)z\otimes xz].
 \end{align}
\end{lemma}
\begin{proof}
By Theorem \ref{generatedByDegreeOne},
$\mathrm{gr}\,H\simeq \mathfrak{B}[M\langle(y,xy)\rangle]\# H_8$.
It's straightforward to  prove that $p_1^2$, $p_2^2$ and 
$p_1p_2p_1p_2+p_2p_1p_2p_1$ are primitive elements, so $H\simeq \mathrm{gr}\,H$.
\end{proof}

\begin{lemma}\label{HopfAlg:M2_3(xy,x)}
Suppose $H$ is  a finite  dimensional Hopf algebra with coradical $H_8$ such that  its infinitesimal braiding  is isomorphic to $M\langle(xy,x)\rangle$, then $H\simeq \mathfrak{B}[M\langle(xy,x)\rangle]\# H_8$. 
Let $p_1=[(v_1+v_2)\boxtimes xy]\# 1$, 
$p_2=[(v_1-v_2)\boxtimes x]\# 1$ be a basis of $M\langle(xy,x)\rangle]\#1$,  then $H$ is generated by $x$, $y$, $z$,  $p_1$, $p_2$,  which satisfy the following relations.
\begin{eqnarray}
x^2=y^2=1,\quad z^2=\frac{1}{2}(1+x+y-xy), \quad
xy=yx,\quad zx=yz,\quad zy=xz,\\
 xp_1=p_1x,\quad yp_1=-p_1y,\quad zp_1=p_2 z,\\
 xp_2=-p_2x, \quad
  yp_2=p_2y,\quad zp_2= p_1 xz,\\  
 p_1^2=0,\quad p_2^2=0,\quad 
 p_1p_2p_1p_2+p_2p_1p_2p_1=0.
 \end{eqnarray}
  Its Hopf algebra structure is determined by 
 \begin{align}
 \Delta(p_1)=p_1\otimes 1+xy\otimes p_1,\quad \Delta(p_2)=p_2\otimes 1+x\otimes p_2,\\
 \Delta(x)=x\otimes x,\quad \Delta(y)=y\otimes y,
\quad \Delta(z)=\frac{1}{2}[(1+y)z\otimes z+(1-y)z\otimes xz].
 \end{align}
\end{lemma}
\begin{proof}
By Theorem \ref{generatedByDegreeOne},
$\mathrm{gr}\,H\simeq \mathfrak{B}[M\langle(xy,x)\rangle]\# H_8$.
It's straightforward to  prove that $p_1^2$, $p_2^2$ and 
$p_1p_2p_1p_2+p_2p_1p_2p_1$ are primitive elements, so $H\simeq \mathrm{gr}\,H$. In fact, 
$\mathfrak{B}[M\langle(xy,x)\rangle]$$\# H_8$
is isomorphic to  $\mathfrak{B}[M\langle(y,xy)\rangle]\# H_8$ by Corollary \ref{Isomorphism:B(V)H_8}.
\end{proof}

\begin{proposition}\label{HopfAlg:Omega2}
Suppose $H$ is  a finite  dimensional Hopf algebra with coradical $H_8$ such that its infinitesimal braiding  is isomorphic to $\Omega_2(n_1, n_2)$, then $H\simeq \mathfrak{B}[\Omega_2(n_1, n_2)]\# H_8$. 
Denote 
\begin{align}
p_1=[(v_1+v_2)\boxtimes xy]\# 1,~~  
p_2=[(v_1-v_2)\boxtimes x]\# 1,\quad  v_1, v_2\in V_2,\\
X_j=(v\boxtimes x)\#1, \quad v\in V_1(\mi), \quad j=1,\cdots, n_1,\\
Y_k= (v^\prime\boxtimes x)\#1, \quad v^\prime\in V_1(-\mi),\quad k=1,\cdots, n_2,
\end{align}
then $H$ is generated by $x$,  $y$, $z$, $p_1$, $p_2$, $\{X_j\}_{j=1,\cdots, n_1}$,  $\{Y_k\}_{k=1,\cdots, n_2}$ satisfying the following relations.
 \begin{align}
x^2=y^2=1,\quad z^2=\frac{1}{2}(1+x+y-xy), \\
xy=yx,\quad zx=yz,\quad zy=xz,\\
 xp_1=p_1x,\quad yp_1=-p_1y,\quad zp_1=p_2 z,\\
 xp_2=-p_2x, \quad
  yp_2=p_2y,\quad zp_2= p_1 xz,\\  
 p_1^2=0,\quad p_2^2=0,\quad 
 p_1p_2p_1p_2+p_2p_1p_2p_1=0,\\
  xX_j=-X_jx,\quad yX_j=-yX_j,\quad zX_j=\mi
  X_jxz,\\
   xY_k=-Y_kx,\quad yY_k=-yY_k,\quad zY_k=-\mi
  Y_kxz,\\
   X_{j_1}X_{j_2}+X_{j_2}X_{j_1}=0,\quad j_1,j_2\in\{1,\cdots,n_1\},\\
  Y_{k_1}Y_{k_2}+Y_{k_2}Y_{k_1}=0,\quad  k_1,k_2\in\{1,\cdots, n_2\},\\
  X_j^2=0,\quad Y_k^2=0,\quad X_jY_k+Y_kX_j=0,  \\
    p_2X_j+X_jp_2=0,
 \quad p_2Y_k+Y_kp_2=0,\label{ideal:xxM(2_3xy,x)1}\\
 p_1X_j-X_jp_1=0, \quad
 p_1Y_k-Y_kp_1=0.\label{ideal:xxM(2_3xy,x)2}
 \end{align}
 Its Hopf algebra structure is determined by 
  \begin{align}
   \Delta(p_1)=p_1\otimes 1+xy\otimes p_1,\quad \Delta(p_2)=p_2\otimes 1+x\otimes p_2,\\
   \Delta(X_j)=X_j\otimes 1+x\otimes X_j, \quad
 \Delta(Y_k)=Y_k\otimes 1+x\otimes Y_k,\\
 \Delta(x)=x\otimes x,\quad \Delta(y)=y\otimes y,
\quad \Delta(z)=\frac{1}{2}[(1+y)z\otimes z+(1-y)z\otimes xz].
 \end{align}
\end{proposition}
\begin{proof}
By Theorem \ref{generatedByDegreeOne},
$\mathrm{gr}\,H\simeq \mathfrak{B}[\Omega_2(n_1,n_2)]\# H_8$. By Lemma \ref{HopfAlg:M2_3(xy,x)} and Proposition \ref{HopfAlge:A_1}, 
we only need to prove that the lifting relations \eqref{ideal:xxM(2_3xy,x)1} and \eqref{ideal:xxM(2_3xy,x)2} are  only possible by the given generators, which can be obtained  from the following formulae
\begin{align*}
x(p_1X_j-X_jp_1)&=-(p_1X_j-X_jp_1)x,\quad x(p_1Y_k-Y_kp_1)=-(p_1Y_k-Y_kp_1),\\
\Delta(p_1X_j-X_jp_1)&=(p_1X_j-X_jp_1)\otimes 1+y
\otimes (p_1X_j-X_jp_1),\\
\Delta(p_1Y_k-Y_kp_1)&=(p_1Y_k-Y_kp_1)\otimes 1+y
\otimes (p_1Y_k-Y_kp_1),\\
\Delta(p_2X_j+X_jp_2)&=(p_2X_j+X_jp_2)\otimes 1+1\otimes (p_2X_j+X_jp_2),\\
\Delta(p_2Y_k+Y_kp_2)&=(p_2Y_k+Y_kp_2)\otimes 1+1\otimes (p_2Y_k+Y_kp_2).
\end{align*}
So $H\simeq \mathrm{gr}\, H$. 
\end{proof}

\begin{definition}\label{definition:A_6}
For  $\lambda \in \mathbb{K}$, denote by $\mathfrak{A}_6(\lambda)$ the algebra generated by $x$, $y$, $z$, $p_1$, $p_2$, $q_1$, $q_2$
satisfying the following relations
\begin{eqnarray}
x^2=y^2=1,\quad z^2=\frac{1}{2}(1+x+y-xy), \\
xy=yx,\quad zx=yz,\quad zy=xz,\\
 xp_1=p_1x,\quad yp_1=-p_1y,\quad zp_1=p_2 z,\\
  xp_2=-p_2x,\quad
  yp_2=p_2y,\quad zp_2= p_1 xz,\\
 p_1^2=0,\quad p_2^2=0,\quad 
 p_1p_2p_1p_2+p_2p_1p_2p_1=0,
 \label{ideal:TwoM2_3_1}\\
 xq_1=q_1x,\quad yq_1=-q_1y,\quad zq_1=q_2 z,\\
 xq_2=-q_2x, \quad
  yq_2=q_2y,\quad zq_2= q_1 xz,\\  
 q_1^2=0,\quad q_2^2=0,\quad 
 q_1q_2q_1q_2+q_2q_1q_2q_1=0,\label{ideal:TwoM2_3_2}\\
 p_1q_1+q_1p_1=\lambda(1-x),\quad 
p_2q_2+q_2p_2=\lambda(1-y), \label{ideal:TwoM2_3_3}\\
 p_1q_2-q_2p_1=0,\quad p_2q_1+q_1p_2=0.\label{ideal:TwoM2_3_4}
 \end{eqnarray}
 It is a Hopf algebra with its structure determined by 
\begin{align}
\Delta(x)=x\otimes x,\quad \Delta(y)=y\otimes y,
\quad \Delta(z)=\frac{1}{2}[(1+y)z\otimes z+(1-y)z\otimes xz],\\
\Delta(p_1)=p_1\otimes 1+y\otimes p_1,\quad \Delta(p_2)=p_2\otimes 1+xy\otimes p_2,\\
\Delta(q_1)=q_1\otimes 1+xy\otimes q_1,\quad \Delta(q_2)=q_2\otimes 1+x\otimes q_2.
\end{align}
\end{definition}
\begin{remark}
In fact, $\mathfrak{A}_6(\lambda)\simeq [T(\Omega_6)\#H_8]/{\mathcal{I}(\lambda)}$, where ${\mathcal{I}(\lambda)}$ is a Hopf ideal generated by relations \eqref{ideal:TwoM2_3_1}, \eqref{ideal:TwoM2_3_2},
\eqref{ideal:TwoM2_3_3} and \eqref{ideal:TwoM2_3_4}.
It's obvious that $\mathfrak{A}_6(\lambda)$ is finite-dimensional.  
\end{remark}
\begin{proposition}\label{HopfAlg:A_6}
Suppose $H$ is  a finite  dimensional Hopf algebra with coradical $H_8$ such that  its infinitesimal braiding  is isomorphic to $\Omega_6$, then $H\simeq \mathfrak{A}_6(\lambda)$.
\end{proposition}
\begin{proof}
By Theorem \ref{generatedByDegreeOne},  $\mathrm{gr}\,H\simeq \mathfrak{B}(\Omega_6)\#H_8$. We can suppose that $H$ is generated by generators $x$, $y$, $z$ in $H_8$ and $p_1=[(v_1+v_2)\boxtimes y]\# 1$,  
$p_2=[(v_1-v_2)\boxtimes xy]\# 1$, 
$q_1=[(v_1+v_2)\boxtimes xy]\# 1$, 
$q_2=[(v_1-v_2)\boxtimes x]\# 1$ in $[M\langle(y,xy)\rangle\oplus M\langle(xy,x)\rangle]\#1$. 
It's easy to see that formulae above in Definition \ref{definition:A_6} except 
\eqref{ideal:TwoM2_3_3} and \eqref{ideal:TwoM2_3_4} hold in $H$ from the bosonization $\mathfrak{B}(\Omega_6)\#H_8$ and Lemma \ref{HopfAlg:M2_3(y,xy)}, \ref{HopfAlg:M2_3(xy,x)}.
Since $\mathrm{gr}\,[T(\Omega_6)\#H_8]/{\mathcal{I}(\lambda)}\simeq \mathfrak{B}(\Omega_6)\#H_8$, it's enough to prove that  
\eqref{ideal:TwoM2_3_3} and \eqref{ideal:TwoM2_3_4} are the only possible lifting relations by the given generators.

Since $r=0$ in $\mathrm{gr}\,H$ for $r=p_1q_1+q_1p_1$, 
$p_2q_2+q_2p_2$, 
 $p_1q_2-q_2p_1$, $p_2q_1+q_1p_2$, we have $r\in H_8\oplus \mathbb{K}\left(\Omega_5\#1\right)$.
It's only possible that  
$p_1q_1+q_1p_1=\lambda_1(1-x)$, 
$p_2q_2+q_2p_2=\lambda_2(1-y)$, 
 $p_1q_2-q_2p_1=\lambda_3(1-xy)$, $p_2q_1+q_1p_2=0$ for $\lambda_1$, $\lambda_2$ and $\lambda_3$ in $\mathbb{K}$,   because 
\begin{align*}
\Delta(p_1q_1+q_1p_1)=(p_1q_1+q_1p_1)\otimes 1+x\otimes 
(p_1q_1+q_1p_1),\\
\Delta(p_2q_2+q_2p_2)=(p_2q_2+q_2p_2)\otimes 1+
y\otimes (p_2q_2+q_2p_2),\\
\Delta(p_1q_2-q_2p_1)=(p_1q_2-q_2p_1)\otimes 1+xy
\otimes (p_1q_2-q_2p_1),\\
\Delta(p_2q_1+q_1p_2)=(p_2q_1+q_1p_2)\otimes 1+
1\otimes (p_2q_1+q_1p_2).
\end{align*}
Since $z(p_1q_1+q_1p_1)=(p_2q_2+q_2p_2)z$ and 
$z(p_1q_2-q_2p_1)=(p_2q_1+q_1p_2)xz$, we have $\lambda_1=\lambda_2$ and $\lambda_3=0$. 
\end{proof}

\begin{lemma}\label{HopfAlg:W(b_1,-1)}
Suppose $H$ is a finite  dimensional Hopf algebra with coradical $H_8$ such that its infinitesimal braiding  is isomorphic to $W^{b_1,-1}$,  where $b_1=\pm 1$.
 Then there exist  parameters $\lambda_1$ and $\lambda_2$ in $\mathbb{K}$ such that $H$ is generated by $x$, $y$, $z$, $p_1$, $p_2$, which satisfy the following relations.
\begin{eqnarray}
x^2=y^2=1,\quad z^2=\frac{1}{2}(1+x+y-xy), \\
xy=yx,\quad zx=yz,\quad zy=xz,\\
xp_1=p_1x,\quad  yp_1=p_1y,\quad
xp_2=-p_2x,\quad  yp_2=-p_2y,\label{HopfAlg:W(b_1,-1)0}\\ 
zp_1=-p_1 z,\quad 
zp_2=\mi b_1 p_2 xz,\label{HopfAlg:W(b_1,-1)1}\\
p_1^2=\lambda_1(1-xy), \quad p_2^2=\mi b_1\lambda_1(1-xy),\label{ideal:W(b_1,-1)1}\\
p_1p_2p_1p_2+p_2p_1p_2p_1=\lambda_2(1-xy).\label{ideal:W(b_1,-1)2}
\end{eqnarray}
Its Hopf algebra structure is determined by
\begin{align}
\Delta(x)&=x\otimes x,\quad \Delta(y)=y\otimes y,
\quad \Delta(z)=\frac{1}{2}[(1+y)z\otimes z+(1-y)z\otimes xz],\\
\Delta(p_1)&=\left[f_{00}-
\mi b_1f_{11}\right]z\otimes p_1+\left[f_{10}+
\mi b_1f_{01}\right]z\otimes p_2+ 
p_1\otimes 1,
\label{Delta:W(b_1-1)p1}\\
\Delta(p_2)&=\left[f_{00}+
\mi b_1f_{11}\right]z\otimes p_2+
\left[f_{10}-
\mi b_1f_{01}\right]z\otimes p_1+p_2\otimes 1.\label{Delta:W(b_1-1)p2}
\end{align}
\end{lemma}
\begin{remark}
In fact,  $H\simeq \left[T\left(W^{b_1,-1}\right)\#H_8\right]\Big/{\mathcal{I}(\lambda_1,\lambda_2)}$, where ${\mathcal{I}(\lambda_1,\lambda_2)}$ is a Hopf ideal generated by \eqref{ideal:W(b_1,-1)1} and \eqref{ideal:W(b_1,-1)2}. It's obvious that $H$ is finite dimensional. 
\end{remark}
\begin{proof}
By Theorem \ref{generatedByDegreeOne}, 
 $\mathrm{gr}\,H\simeq \mathfrak{B}(W^{b_1,-1})\#H_8$. We can suppose $H$ is generated by $x$, $y$, $z$ and $p_1$, $p_2$
 with $x$, $y$, $z$ in $H$ and $p_1=\left(w_1^{b_1,-1}+\mi b_1w_2^{b_1,-1}\right)\#1$,   $p_2=\left(w_1^{b_1,-1}-\mi b_1w_2^{b_1,-1}\right)\#1$.
 
Formulae \eqref{HopfAlg:W(b_1,-1)0}, \eqref{HopfAlg:W(b_1,-1)1}, 
 \eqref{Delta:W(b_1-1)p1} and  \eqref{Delta:W(b_1-1)p2} hold in $H$
 by a  straightforward computation for the bosonization $\mathfrak{B}(W^{b_1,-1})\#H_8$. Since 
\begin{align*}
\Delta(p_1^2)
&=\frac{1}{2}(1+xy)\otimes p_1^2+
\frac{\mi b_1}{2}(1-xy)\otimes p_2^2+ 
p_1^2\otimes 1,\\
\Delta(p_2^2)
&=\frac{1}{2}(1+xy)\otimes p_2^2-\frac{\mi b_1}{2}(1-xy)\otimes p_1^2+ 
p_2^2\otimes 1, 
\end{align*}
there must exist a parameter $\lambda_1\in\mathbb{K}$ such that
$p_1^2=\lambda_1(1-xy)$ and
$p_2^2=\mi b_1 \lambda_1(1-xy)$.
\begin{align*}
\Delta(p_1 p_2)
&=\frac{1}{2}(x+y)\otimes p_1 p_2
+\frac{\mi b_1}{2}(x-y)\otimes p_2 p_1+
p_1 p_2\otimes 1+
\\&\quad +
p_2\left[f_{00}-
\mi b_1f_{11}\right]z\otimes p_1+
p_2\left[f_{10}+
\mi b_1f_{01}\right]z\otimes p_2+ \\
&\quad + p_1\left[f_{00}+
\mi b_1f_{11}\right]z\otimes p_2+
p_1\left[f_{10}-
\mi b_1f_{01}\right]z\otimes p_1,
\end{align*}
\begin{align*}
\Delta( p_2 p_1)
&=\frac{1}{2}(x+y)\otimes p_2 p_1+
\frac{\mi b_1}{2}(y-x)\otimes p_1 p_2+ 
p_2 p_1\otimes 1+\\
&\quad + \left[f_{00}+
\mi b_1f_{11}\right]zp_1\otimes p_2+
\left[f_{10}-
\mi b_1f_{01}\right]zp_1\otimes p_1+\\
& \quad +p_2\left[f_{00}-
\mi b_1f_{11}\right]z\otimes p_1+
p_2\left[f_{10}+
\mi b_1f_{01}\right]z\otimes p_2.
\end{align*}
Denote  $\Delta (p_1 p_2 )=B-A+E_1$ and 
$\Delta(p_2 p_1 )=B+A+E_2$, where 
\begin{align*}
A&=\left[f_{00}+
\mi b_1f_{11}\right]zp_1\otimes p_2+
\left[f_{10}-\mi b_1f_{01}\right]zp_1\otimes p_1,\\
B&=p_2\left[f_{00}-
\mi b_1f_{11}\right]z\otimes p_1+
p_2\left[f_{10}+\mi b_1f_{01}\right]z\otimes p_2,\\
E_2&=\frac{1}{2}(x+y)\otimes p_2 p_1+
\frac{\mi b_1}{2}(y-x)\otimes p_1 p_2+ 
p_2 p_1\otimes 1,\\
E_1&=\frac{1}{2}(x+y)\otimes p_1 p_2
+\frac{\mi b_1}{2}(x-y)\otimes p_2 p_1+
p_1 p_2\otimes 1.
\end{align*}
We can obtain $A^2+B^2=0$, since 
\begin{align*}
A^2
&=-\frac{1}{2}(x+y)p_1^2\otimes p_2^2
+\frac{\mi b_1}{2}(1-xy)p_1^2\otimes p_1^2
=\mi b_1\lambda_1^2(1-xy)\otimes (1-xy),
\end{align*}
\begin{align*}
B^2
&=- \frac{1}{2}(x+y) p_2^2
\otimes p_1^2+
 \frac{\mi b_1}{2}(1-xy) p_2^2
\otimes p_2^2
=-\mi b_1\lambda_1^2(1-xy)\otimes (1-xy).
\end{align*}
Keeping in mind that
\begin{align*}
p_1(p_1p_2+p_2p_1)&=(p_2p_1+p_1p_2)p_1,\quad 
p_2(p_1p_2+p_2p_1)=(p_2p_1+p_1p_2)p_2,\\
p_1(p_1p_2-p_2p_1)&=(p_2p_1-p_1p_2)p_1,\quad 
p_2(p_1p_2-p_2p_1)=(p_2p_1-p_1p_2)p_2,\\
(x+y)p_2(f_{00}&-\mi b_1f_{11})z=-p_2(f_{00}-\mi b_1f_{11})z(x+y),\\
(x-y)p_2(f_{00}&-\mi b_1f_{11})z=p_2(f_{00}-\mi b_1f_{11})z(x-y),\\
(x+y)p_2(f_{10}&+\mi b_1f_{01})z=-p_2(f_{10}+\mi b_1f_{01})z(x+y),\\
(x-y)p_2(f_{10}&+\mi b_1f_{01})z=p_2(f_{10}+\mi b_1f_{01})z(x-y),\\
(p_1p_2+p_2p_1)&p_2(f_{00}-\mi b_1f_{11})z=-p_2(f_{00}-\mi b_1f_{11})z(p_1p_2+p_2p_1),\\
(p_1p_2+p_2p_1)&p_2(f_{10}+\mi b_1f_{01})z=-p_2(f_{10}+\mi b_1f_{01})z(p_1p_2+p_2p_1),
\end{align*}
we deduce $B(E_1+E_2)
+(E_1+E_2)B=0$.  Similarly, we have  $A(E_2-E_1)+(E_2-E_1)A=0$.
\begin{align*}
&\quad \Delta(p_1 p_2 p_1 p_2
+p_2 p_1 p_2 p_1)
=(B-A+E_1)^2+(B+A+E_2)^2\\
&=2(A^2+B^2)+B(E_1+E_2)
+(E_1+E_2)B+A(E_2-E_1)+(E_2-E_1)A+E_1^2+E_2^2\\
&=E_1^2+E_2^2=\left(
\frac{1}{2}(x+y)\otimes p_1 p_2
+\frac{\mi b_1}{2}(x-y)\otimes p_2 p_1+
p_1 p_2\otimes 1\right)^2+\\
&\quad +\left(
\frac{1}{2}(x+y)\otimes p_2 p_1+
\frac{\mi b_1}{2}(y-x)\otimes p_1 p_2+ 
p_2 p_1\otimes 1\right)^2\\
&=\frac{1}{2}(1+xy)\otimes \left(p_1 p_2\right)^2
-\frac{1}{2}(1-xy)\otimes \left(p_2 p_1\right)^2
+\left(p_1 p_2\right)^2\otimes 1+\\
&\quad+
\frac{1}{2}(1+xy)\otimes \left(p_2 p_1\right)^2
-\frac{1}{2}(1-xy)\otimes \left(p_1 p_2\right)^2
+\left(p_2 p_1\right)^2\otimes 1\\
&=xy\otimes \left[\left(p_1 p_2\right)^2+\left(p_2 p_1\right)^2\right]
+\left[\left(p_1 p_2\right)^2+\left(p_2 p_1\right)^2\right]\otimes 1.
\end{align*}
So there exists a parameter $\lambda_2\in\mathbb{K}$ such that 
$p_1p_2p_1p_2+p_2p_1p_2p_1=\lambda_2(1-xy)$.

We have
$H\simeq \left[T(W^{b_1,-1})\#H_8\right]\Big/{\mathcal{I}(\lambda_1,\lambda_2)}$, because $\mathrm{gr}\,\left\{\left[T(W^{b_1,-1})\#H_8\right]\Big/{\mathcal{I}(\lambda_1,\lambda_2)}\right\}\simeq \mathfrak{B}\left(W^{b_1,-1}\right)\#H_8$.
\end{proof}

\begin{definition}\label{Definition:HopfAlgA_7}
For a set  of  parameters $I_7=\{\lambda_j\in \mathbb{K}|j=1,\cdots, 5\}$, denote by $\mathfrak{A}_7(I_7)$ the algebra generated by $x$, $y$, $z$, 
$p_1$, $p_2$, $q_1$, $q_2$ satisfying the following relations
\begin{eqnarray}
x^2=y^2=1,\quad z^2=\frac{1}{2}(1+x+y-xy), \\
xy=yx,\quad zx=yz,\quad zy=xz,\\
xp_1=p_1x,\quad  yp_1=p_1y,\quad
xp_2=-p_2x,\quad  yp_2=-p_2y,\\ 
zp_1=-p_1 z,\quad 
zp_2=\mi  p_2 xz,\\
p_1^2=\lambda_1(1-xy), \quad p_2^2=\mi \lambda_1(1-xy),
\label{ideal:TwoW(pm 1,-1)1}\\
p_1p_2p_1p_2+p_2p_1p_2p_1=\lambda_2(1-xy),
\label{ideal:TwoW(pm 1,-1)2}\\
xq_1=q_1x,\quad  yq_1=q_1y,\quad
xq_2=-q_2x,\quad  yq_2=-q_2y,\\ 
zq_1=-q_1 z,\quad 
zq_2=-\mi  q_2 xz,\\
q_1^2=\lambda_3(1-xy), \quad q_2^2=-\mi \lambda_3(1-xy),
\label{ideal:TwoW(pm 1,-1)3}\\
q_1q_2q_1q_2+q_2q_1q_2q_1=\lambda_4(1-xy),
\label{ideal:TwoW(pm 1,-1)4}\\
p_1q_2+q_2p_1=0,\quad p_2q_1+q_1p_2=0,
\label{ideal:TwoW(pm 1,-1)5}\\
p_1q_1+q_1p_1=\lambda_5(x+y-2),\quad 
p_2q_2-q_2p_2=-\mi \lambda_5(x-y).
\label{ideal:TwoW(pm 1,-1)6}
\end{eqnarray}
 It is a Hopf algebra with its structure determined by 
 \begin{align}
 \Delta(p_1)&=\left[f_{00}-
\mi f_{11}\right]z\otimes p_1+
\left[f_{10}+\mi f_{01}\right]z\otimes p_2+ 
p_1\otimes 1,\\
\Delta(p_2)&=\left[f_{00}+
\mi f_{11}\right]z\otimes p_2+
\left[f_{10}-\mi f_{01}\right]z\otimes p_1+p_2\otimes 1,\\
\Delta(q_1)&=\left[f_{00}+\mi f_{11}\right]z\otimes q_1+
\left[f_{10}-\mi f_{01}\right]z\otimes q_2+ 
q_1\otimes 1,\\
\Delta(q_2)&=\left[f_{00}-\mi f_{11}\right]z\otimes q_2+
\left[f_{10}+\mi f_{01}\right]z\otimes q_1+q_2\otimes 1,\\
\Delta(x)&=x\otimes x,\quad \Delta(y)=y\otimes y,
\quad \Delta(z)=\frac{1}{2}[(1+y)z\otimes z+(1-y)z\otimes xz].
\end{align}
\end{definition}
\begin{remark}
In fact, $\mathfrak{A}_7(I_7)\simeq\left[T\left(W^{1,-1}\oplus W^{-1,-1}\right) \# H_8\right]\Big/{\mathcal{I}(I_7)}$,  where $\mathcal{I}(I_7)$ is a Hopf ideal generated by relations \eqref{ideal:TwoW(pm 1,-1)1}, 
\eqref{ideal:TwoW(pm 1,-1)2}, \eqref{ideal:TwoW(pm 1,-1)3}, 
\eqref{ideal:TwoW(pm 1,-1)4}, \eqref{ideal:TwoW(pm 1,-1)5}
and \eqref{ideal:TwoW(pm 1,-1)6}.
Since \begin{align*}
p_1p_2p_1p_2p_1&=p_1[\lambda_2(1-xy)-p_1p_2p_1p_2]=
[\lambda_2p_1-\lambda_1p_2p_1p_2](1-xy),\\
p_2p_1p_2p_1p_2&=p_2[\lambda_2(1-xy)-p_2p_1p_2p_1]=
[\lambda_2p_2-\lambda_1\mi p_1p_2p_1](1-xy),
\end{align*}
$\dim\left<1, p_1, p_2\right><\infty$ for the subalgebra $\left<1, p_1, p_2\right>$ generated by $1$, $p_1$, $p_2$. Similarly, we have $\dim\left<1, q_1, q_2\right><\infty$. We can deduce  that $$\dim A_7(I_7)=\dim\left<1, p_1, p_2\right> \dim\left<1, q_1, q_2\right> \dim H_8<\infty. $$
\end{remark}
\begin{proposition}\label{HopfAlg:A_7}
Suppose $H$ is a finite  dimensional Hopf algebra with coradical $H_8$ such that its infinitesimal braiding  is isomorphic to $\Omega_7$, then $H\simeq \mathfrak{A}_7(I_7)$.
\end{proposition}
\begin{proof}
By Theorem \ref{generatedByDegreeOne}, we have $\mathrm{gr}\,H\simeq \mathfrak{B}(\Omega_7)\#H_8$. We can suppose $H$ is generated by $x$, $y$, $z$, 
$p_1$, $p_2$, $q_1$, $q_2$ with $x$, $y$, $z\in H_8$ and 
 \begin{align}
p_1&=\left(w_1^{1,-1}+\mi w_2^{1,-1}\right)\#1, 
&p_2&=\left(w_1^{1,-1}-\mi w_2^{1,-1}\right)\#1,\\ 
q_1&=\left(w_1^{-1,-1}-\mi w_2^{-1,-1}\right)\#1,  
&q_2&=\left(w_1^{-1,-1}+\mi w_2^{-1,-1}\right)\#1.
\end{align}
By lemma \ref{HopfAlg:W(b_1,-1)}, we only need to prove that 
\eqref{ideal:TwoW(pm 1,-1)5}
and \eqref{ideal:TwoW(pm 1,-1)6} hold in $H$.
It's only possible for  $p_1q_2+q_2p_1=0$,  $p_2q_1+q_1p_2=0$, since $x(p_1q_2+q_2p_1)=-(p_1q_2+q_2p_1)x$, $
x(p_2q_1+q_1p_2)=-(p_2q_1+q_1p_2)x$, and 
\begin{align*}
\Delta(p_1q_2+q_2p_1)
&=\frac{1}{2}\left[(1+xy)+\mi (1-xy)\right]\otimes (p_1q_2+q_2p_1)+
(p_1q_2+q_2p_1)\otimes 1,\\
\Delta(p_2q_1+q_1p_2)
&=\frac{1}{2}\left[(1+xy)-\mi(1-xy)\right]\otimes (p_2q_1+q_1p_2)+
(p_2q_1+q_1p_2)\otimes 1,
\end{align*}
Similarly, we can get  \eqref{ideal:TwoW(pm 1,-1)6}, since 
\begin{align*}
z(p_1q_1+q_1p_1)&=(p_1q_1+q_1p_1)z, \quad 
z(p_2q_2-q_2p_2)=-(p_2q_2-q_2p_2)z,\\
\Delta(p_1q_1+q_1p_1)
&=\frac{x+y}{2}
\otimes (p_1q_1+q_1p_1)+ (p_1q_1+q_1p_1)\otimes 1
+\frac{\mi(x-y)}{2}\otimes (p_2q_2-
 q_2p_2),\\
 \Delta(p_2q_2-q_2p_2)
 &=\frac{x+y}{2}\otimes (p_2q_2-q_2p_2)+(p_2q_2-q_2p_2)\otimes 1-\frac{\mi(x-y)}{2}\otimes 
(p_1q_1+q_1p_1).
\end{align*}
We have $H\simeq  \mathfrak{A}_7(I_7)$,  because $\mathrm{gr}\,\{\left[T\left(\Omega_7\right) \# H_8\right]\big/\mathcal{I}(I_7)\}\simeq \mathfrak{B}(\Omega_7)\#H_8$.
\end{proof}

\begin{definition}\label{Definition:HopfAlgA_4}
For a set  of  parameters $I_4=\{\lambda_1, \lambda_2, \lambda_{j,k}\in \mathbb{K}|j=1,\cdots, n_1, k=1,\cdots, n_2\}$,
denote by $\mathfrak{A}_4(n_1,n_2;I_4)$ the algebra generated by $x$, $y$, $z$, $p_1$, $p_2$, 
 $\{X_j\}_{j=1,\cdots, n_1}$,  $\{Y_k\}_{k=1,\cdots, n_2}$ satisfying the following relations
  \begin{align}
  x^2=y^2=1,\quad z^2=\frac{1}{2}(1+x+y-xy), \label{formulae:A_4_1}\\
xy=yx,\quad zx=yz,\quad zy=xz,\\
xp_1=p_1x,\quad  yp_1=p_1y,\quad
xp_2=-p_2x,\quad  yp_2=-p_2y,\label{FGK1}\\ 
zp_1=-p_1 z,\quad 
zp_2=\mi  p_2 xz,\\
p_1^2=\lambda_1(1-xy), \quad p_2^2=\mi \lambda_1(1-xy),
\label{ideal:xyW(b1,-1)1}\\
p_1p_2p_1p_2+p_2p_1p_2p_1=\lambda_2(1-xy),\label{ideal:xyW(b1,-1)2}\\
  xX_j=-X_jx,\quad yX_j=-X_jy,\quad zX_j=\mi X_j xz,\\
  xY_k=-Y_kx,\quad yY_k=-Y_ky,\quad zY_k=\mi Y_k xz,\\
  X_{j_1}^2=0,\quad X_{j_1}X_{j_2}+X_{j_2}X_{j_1}=0,\quad j_1,j_2\in\{1,\cdots,n_1\},\label{ideal:xyW(b1,-1)3}\\
  Y_{k_1}^2=0,\quad Y_{k_1}Y_{k_2}+Y_{k_2}Y_{k_1}=0,\quad  k_1,k_2\in\{1,\cdots, n_2\},\label{ideal:xyW(b1,-1)4}\\
  X_jY_k+Y_kX_j=\lambda_{j,k}(1-xy), \label{ideal:xyW(b1,-1)5}
 \end{align}
 \begin{align}
p_1Y_k-Y_kp_1=0
, \quad
p_2Y_k+Y_kp_2=0,\quad
p_1X_j-X_jp_1=0, \quad
p_2X_j+X_jp_2=0.\label{ideal:xyW(b1,-1)6}
\end{align}
 It is a Hopf algebra with its structure determined by 
 \begin{align}
 \Delta(X_j)&=X_j\otimes 1+x\otimes X_j, \quad 
 \Delta(Y_k)=Y_k\otimes 1+y\otimes Y_k,  \label{coproduct:A_4_1}\\
\Delta(p_1)&=(f_{00}-\mi  f_{11})z\otimes p_1+
(f_{10}+\mi  f_{01})z\otimes p_2+p_1\otimes 1,\\
\Delta(p_2)&=(f_{00}+\mi  f_{11})z\otimes p_2+
(f_{10}-\mi  f_{01})z\otimes p_1+p_2\otimes 1,\\
\Delta(x)&=x\otimes x,\quad \Delta(y)=y\otimes y,
\quad \Delta(z)=\frac{1}{2}[(1+y)z\otimes z+(1-y)z\otimes xz].
\label{coproduct:A_4_2}
\end{align}
\end{definition}
\begin{remark}
We can observe  that 
$$\dim \mathfrak{A}_4(n_1,n_2;I_4)=\dim\left<1,p_1,p_2\right>
\dim\left<1, \{X_j\}_{j=1,\cdots, n_1}  \right>\dim\left<1, \{Y_k\}_{k=1,\cdots, n_2}\right>\dim H_8<\infty$$
for subalgebra $\left<1,p_1,p_2\right>$ generated by $1$, $p_1$, $p_2$, subalgebra $\left<1, \{X_j\}_{j=1,\cdots, n_1}  \right>$ generated by $1$, $\{X_j\}_{j=1,\cdots, n_1}$, and subalgebra $\left<1, \{Y_k\}_{k=1,\cdots, n_2}\right>$ generated by $1$, $\{Y_k\}_{k=1,\cdots, n_2}$.

In fact, $\mathfrak{A}_4(n_1,n_2;I_4)\simeq T[\Omega_4(n_1,n_2)]\#H_8/{\mathcal{I}(I_4)}$, where $\mathcal{I}(I_4)$ is a Hopf ideal genereated by relations 
 \eqref{ideal:xyW(b1,-1)1}, \eqref{ideal:xyW(b1,-1)2}, \eqref{ideal:xyW(b1,-1)3}, \eqref{ideal:xyW(b1,-1)4}, \eqref{ideal:xyW(b1,-1)5}, \eqref{ideal:xyW(b1,-1)6}.
\end{remark}

\begin{proposition}\label{HopfAlg:Omega_4}
Suppose $H$ is a finite  dimensional Hopf algebra with coradical $H_8$ such  that its infinitesimal braiding  is isomorphic to $\Omega_4(n_1,n_2)$, then $H\simeq \mathfrak{A}_4(n_1,n_2; I_4)$.  
\end{proposition}
\begin{proof}
By Theorem \ref{generatedByDegreeOne}, we have $\mathrm{gr}\,H\simeq \mathfrak{B}[\Omega_4(n_1,n_2)]\#H_8$. We can suppose $H$ is generated by $x$, $y$, $z$, 
$p_1$, $p_2$, $\{X_j\}_{j=1,\cdots, n_1}$,  $\{Y_k\}_{k=1,\cdots, n_2}$ with $x$, $y$, $z\in H_8$ and 
\begin{align}
p_1=\left(w_1^{b_1,-1}+\mi b_1w_2^{b_1,-1}\right)\#1, \quad 
p_2=\left(w_1^{b_1,-1}-\mi b_1w_2^{b_1,-1}\right)\#1,\\
X_j=(v\boxtimes x)\#1,  \quad j=1,\cdots, n_1,\\
Y_k= (v\boxtimes y)\#1, \quad v\in V_1(\mi),\quad k=1,\cdots, n_2.
\end{align}
As similarly proved in Proposition \ref{HopfAlge:A_1} and Lemma 
\ref{HopfAlg:W(b_1,-1)}, formulae 
\eqref{formulae:A_4_1}--\eqref{ideal:xyW(b1,-1)5} and \eqref{coproduct:A_4_1}--\eqref{coproduct:A_4_2} hold in $H$.
Since $r=0$ in $\mathrm{gr}\,H$ for $r=p_1Y_k-Y_kp_1$ and $p_2Y_k+Y_kp_2$, $r$ is an element of  at most degree one. It's only possible for 
\begin{align*}
p_1Y_k-Y_kp_1=-\mu_k\left(-f_{10}+\mi b_1f_{01}\right)z
, \quad
p_2Y_k+Y_kp_2=-\mu_k\left(f_{00}-\mi b_1f_{11}\right)z+
\mu_k 1,
\end{align*}
because of the following relations
\begin{align*}
x\left(p_1Y_k-Y_kp_1\right)&=-\left(p_1Y_k-Y_kp_1\right)x,\quad
 z\left(p_1Y_k-Y_kp_1\right)=-b_1\mi \left(p_1Y_k-Y_kp_1\right)xz,\\
 x\left(p_2Y_k+Y_kp_2\right)&=\left(p_2Y_k+Y_kp_2\right)x,\quad
 z\left(p_2Y_k+Y_kp_2\right)=\left(p_2Y_k+Y_kp_2\right)z,\\
\Delta(p_1Y_k-Y_kp_1)
&=(p_1Y_k-Y_kp_1)\otimes 1+(f_{00}+\mi b_1 f_{11})z\otimes (p_1Y_k-Y_kp_1)+\\
&\quad+
(-f_{10}+\mi b_1 f_{01})z\otimes (p_2Y_k+Y_kp_2),\\
\Delta(p_2Y_k+Y_kp_2)
&=(p_2Y_k+Y_kp_2)\otimes 1+(f_{00}-\mi b_1 f_{11})z\otimes 
(p_2Y_k+Y_kp_2)-\\
&\quad-(f_{10}+
 \mi b_1 f_{01})z\otimes 
(p_1Y_k-Y_kp_1).
\end{align*}
Similarly, we get 
\begin{align*}
p_1X_j-X_jp_1=-\mu_j^\prime\left(f_{10}-\mi b_1f_{01}\right)z
, \quad
p_2X_j+X_jp_2=-\mu_j^\prime\left(f_{00}-\mi b_1f_{11}\right)z+
\mu_j^\prime 1,
\end{align*}
from the following formulae
\begin{align*}
x\left(p_1X_j-X_jp_1\right)&=-\left(p_1X_j-X_jp_1\right)x,
\quad z\left(p_1X_j-X_jp_1\right)=-b_1\mi \left(p_1X_j-X_jp_1\right)xz,\\
x\left(p_2X_j+X_jp_2\right)&=\left(p_2X_j+X_jp_2\right)x,\quad
z\left(p_2X_j+X_jp_2\right)=\left(p_2X_j+X_jp_2\right)z,\\
\Delta\left(p_1X_j-X_jp_1\right)
&=(f_{00}+\mi b_1 f_{11})z\otimes \left(p_1X_j-X_jp_1\right)+\left(p_1X_j-X_jp_1\right)\otimes 1+\\
&\quad+(f_{10}-\mi b_1 f_{01})z\otimes \left(p_2X_j+X_jp_2\right),\\
\Delta\left(p_2X_j+X_jp_2\right)
&=(f_{00}-\mi b_1 f_{11})z\otimes \left(p_2X_j+X_jp_2\right)+
\left(p_2X_j+X_jp_2\right)\otimes 1+\\
&\quad+(f_{10}+\mi b_1 f_{01})z\otimes \left(p_1X_j-X_jp_1\right).
\end{align*}
Since  $X_jY_k+Y_kX_j=\lambda_{j,k}(1-xy)$, 
$p_1(X_jY_k+Y_kX_j)=(X_jY_k+Y_kX_j)p_1\Rightarrow \mu_j^\prime=\mu_k=0$. So \eqref{ideal:xyW(b1,-1)6} holds in $H$.
We have $H\simeq \mathfrak{A}_4(n_1,n_2;I_4)$ because $ \mathrm{gr}\,\left\{T[\Omega_4(n_1,n_2)]\#H_8/{\mathcal{I}(I_4)}\right\}\simeq \mathfrak{B}[\Omega_4(n_1,n_2)]\#H_8$.
\end{proof}

\subsection*{Proof of Theorem \ref{HopfAlgOverH8}} Let $M$ be one of Yetter-Drinfel'd modules listed in Theorem \ref{NicholsAlg:maintheorem}. We need to give a construction for any finite-dimensional Hopf algebra $H$ over $H_8$ up to isomorphism such that its infinitesimal braiding is isomorphic to $M$. By Theorem \ref{generatedByDegreeOne},
$\mathrm{gr}\,H\simeq \mathfrak{B}(M)\# H_8$.
According to Corollary \ref{Isomorphism:B(V)H_8}, up to isomorphism,  $\mathrm{gr}\,H\simeq \mathfrak{B}(M)\# H_8$ for $M=\Omega_1(n_1,n_2,n_3,n_4)$, $\Omega_2(n_1, n_2)$, 
$\Omega_4(n_1,n_2)$, $\Omega_6$, $\Omega_7$. Proposition \ref{HopfAlge:A_1},  \ref{HopfAlg:Omega2}, \ref{HopfAlg:Omega_4},
\ref{HopfAlg:A_6} and \ref{HopfAlg:A_7} finish the proof.

\section*{Acknowledgements}
I am indebted to my doctoral supervisor Prof. Naihong Hu for his encouragement  and useful discussions which push me to consider the subject of finite dimensional Nichols algebras. I would like to thank my postdoctoral supervisor  Prof. Wenxue Huang,  Yunnan Li and Rongchuan Xiong for useful discussions and constructive comments. Finally I would like to thank the referees for careful reading and useful comments. 
This work was partially supported by
the doctoral research project of Guangdong Province (gdbsh2014004).

\newcommand{\etalchar}[1]{$^{#1}$}
\providecommand{\bysame}{\leavevmode\hbox to3em{\hrulefill}\thinspace}
\providecommand{\MR}{\relax\ifhmode\unskip\space\fi MR }
\providecommand{\MRhref}[2]{%
  \href{http://www.ams.org/mathscinet-getitem?mr=#1}{#2}
}
\providecommand{\href}[2]{#2}


\begin{thebibliography}{CDMM04}

\bibitem[AAH16]{andruskiewitsch2016finite}
N.~Andruskiewitsch, I.~Angiono, and I.~Heckenberger,
  \emph{On finite gk-dimensional nichols algebras over abelian groups}, arXiv
  preprint arXiv:1606.02521 (2016).

\bibitem[AAI]{AAGI}
N.~Andruskiewitsch, I.~Angiono, and A.~G.~Iglesias, \emph{Liftings of {Nichols} algebras of diagonal type {I}. {Cartan}
  type {A}.}, arXiv:1509.01622, Int. Math. Res. Notices, to appear.

\bibitem[AAI{\etalchar{+}}14]{MR3133699}
N.~Andruskiewitsch, I.~Angiono, A.~G.~Iglesias, A.~Masuoka, and C.~Vay, \emph{Lifting via cocycle
  deformation}, J. Pure Appl. Algebra \textbf{218} (2014), no.~4, 684--703.
  \MR{3133699}

\bibitem[ACG15]{MR3395052}
N.~Andruskiewitsch, G.~Carnovale, and G.~A.~Garc{\'{\i}}a, \emph{Finite-dimensional pointed {H}opf algebras over finite
  simple groups of {L}ie type {I}. {N}on-semisimple classes in {${\bf
  PSL}_n(q)$}}, J. Algebra \textbf{442} (2015), 36--65. \MR{3395052}

\bibitem[ACG16]{MR3493214}
\bysame, \emph{Finite-dimensional pointed {H}opf algebras over finite simple
  groups of {L}ie type {II}: unipotent classes in symplectic groups}, Commun.
  Contemp. Math. \textbf{18} (2016), no.~4, 1550053, 35. \MR{3493214}

\bibitem[AD05]{MR2136919}
N.~Andruskiewitsch and S.~D{\u{a}}sc{\u{a}}lescu, \emph{On finite
  quantum groups at {$-1$}}, Algebr. Represent. Theory \textbf{8} (2005),
  no.~1, 11--34. \MR{2136919}

\bibitem[AFGV10]{andruskiewitsch2010pointed}
N.~Andruskiewitsch, F.~Fantino, M.~Gra{\~n}a, and L.~Vendramin, \emph{Pointed
  {H}opf algebras over some sporadic simple groups}, C. R. Math. Acad. Sci.
  Paris \textbf{348} (2010), no.~11-12, 605--608. \MR{2652482}

\bibitem[AFGV11]{andruskiewitsch2011finite}
\bysame, \emph{Finite-dimensional pointed {H}opf algebras with alternating
  groups are trivial}, Ann. Mat. Pura Appl. (4) \textbf{190} (2011), no.~2,
  225--245. \MR{2786171}

\bibitem[AHS10]{andruskiewitsch2010nichols}
N.~Andruskiewitsch, I.~Heckenberger, and H.-J.~Schneider, \emph{The {N}ichols algebra of a semisimple {Y}etter-{D}rinfeld
  module}, Amer. J. Math. \textbf{132} (2010), no.~6, 1493--1547. \MR{2766176}

\bibitem[AM16]{2016arXiv160503995A}
G.~{Andres Garcia} and J.~{Matheus Jury Giraldi}, \emph{{On Hopf Algebras over
  quantum subgroups}}, eprint arXiv:1605.03995 (2016).

\bibitem[And02]{MR1907185}
N.~Andruskiewitsch, \emph{About finite dimensional {H}opf algebras},
  Quantum symmetries in theoretical physics and mathematics ({B}ariloche,
  2000), Contemp. Math., vol. 294, Amer. Math. Soc., Providence, RI, 2002,
  pp.~1--57. \MR{1907185 (2003f:16059)}

\bibitem[AS98]{MR1659895}
N.~Andruskiewitsch and H.-J. Schneider, \emph{Lifting of quantum linear spaces
  and pointed {H}opf algebras of order {$p^3$}}, J. Algebra \textbf{209}
  (1998), no.~2, 658--691. \MR{1659895}

\bibitem[AS00]{MR1780094}
N.~Andruskiewitsch and H-J.~Schneider, \emph{Finite quantum
  groups and {C}artan matrices}, Adv. Math. \textbf{154} (2000), no.~1, 1--45.
  \MR{1780094}

\bibitem[AS02]{andruskiewitsch2001pointed}
N.~Andruskiewitsch and H.-J.~Schneider, \emph{Pointed {H}opf
  algebras}, New directions in {H}opf algebras, Math. Sci. Res. Inst. Publ.,
  vol.~43, Cambridge Univ. Press, Cambridge, 2002, pp.~1--68. \MR{1913436}

\bibitem[AS10]{andruskiewitsch2005classification}
\bysame, \emph{On the classification of finite-dimensional pointed {H}opf
  algebras}, Ann. of Math. (2) \textbf{171} (2010), no.~1, 375--417.
  \MR{2630042 (2011j:16058)}

\bibitem[AV11]{MR2863448}
N.~Andruskiewitsch and C.~Vay, \emph{Finite dimensional {H}opf
  algebras over the dual group algebra of the symmetric group in three
  letters}, Comm. Algebra \textbf{39} (2011), no.~12, 4507--4517. \MR{2863448}

\bibitem[BZ10]{MR2661247}
K.~A. Brown and J.~J. Zhang, \emph{Prime regular {H}opf algebras of
  {GK}-dimension one}, Proc. Lond. Math. Soc. (3) \textbf{101} (2010), no.~1,
  260--302. \MR{2661247}

\bibitem[CDMM04]{MR2037722}
C.~C{\u{a}}linescu, S.~D{\u{a}}sc{\u{a}}lescu, A.~Masuoka, and C.~Menini,
  \emph{Quantum lines over non-cocommutative cosemisimple {H}opf algebras}, J.
  Algebra \textbf{273} (2004), no.~2, 753--779. \MR{2037722}

\bibitem[EA03]{MR1987013}
A.~El~Alaoui, \emph{The character table for a {H}opf algebra arising from the
  {D}rinfel'd double}, J. Algebra \textbf{265} (2003), no.~2, 478--495.
  \MR{1987013}

\bibitem[FAGM16]{2016arXiv160806167F}
F.~Fantino, G.~A.~Garcia, and M.~Mastnak, \emph{On
  finite-dimensional copointed {H}opf algebras over dihedral groups}, eprint
  arXiv:1608.06167 (2016).

\bibitem[Fan11]{MR2842083}
Xin Fang, \emph{On defining ideals and differential algebras of {N}ichols
  algebras}, J. Algebra \textbf{346} (2011), 299--331. \MR{2842083}

\bibitem[FG11]{MR2862142}
F.~Fantino and G.~A.~Garcia, \emph{On pointed {H}opf algebras
  over dihedral groups}, Pacific J. Math. \textbf{252} (2011), no.~1, 69--91.
  \MR{2862142}

\bibitem[FGV10]{freyre2007nichols}
S.~Freyre, M.~Gra{\~n}a, and L.~Vendramin, \emph{On
  {N}ichols algebras over {${\rm PGL}(2,q)$} and {${\rm PSL}(2,q)$}}, J.
  Algebra Appl. \textbf{9} (2010), no.~2, 195--208. \MR{2646659 (2011g:16058)}

\bibitem[GGI11]{MR2811166}
G.~A.~Garc{\'{\i}}a and A.~G.~Iglesias,
  \emph{Finite-dimensional pointed {H}opf algebras over {$\mathfrak{S}_4$}},
  Israel J. Math. \textbf{183} (2011), 417--444. \MR{2811166}

\bibitem[GHV11]{grana2011nichols}
M.~Gra{\~n}a, I.~Heckenberger, and L.~Vendramin, \emph{Nichols algebras of
  group type with many quadratic relations}, Adv. Math. \textbf{227} (2011),
  no.~5, 1956--1989. \MR{2803792 (2012f:16077)}

\bibitem[GIV14]{MR3119229}
A.~G.~Iglesias and C.~Vay,
  \emph{Finite-dimensional pointed or copointed {H}opf algebras over affine
  racks}, J. Algebra \textbf{397} (2014), 379--406. \MR{3119229}

\bibitem[Gn00]{MR1779599}
M.~Gra\~na, \emph{A freeness theorem for {N}ichols algebras}, J.
  Algebra \textbf{231} (2000), no.~1, 235--257. \MR{1779599}

\bibitem[Hec06]{MR2207786}
I.~Heckenberger, \emph{The {W}eyl groupoid of a {N}ichols algebra of diagonal
  type}, Invent. Math. \textbf{164} (2006), no.~1, 175--188. \MR{2207786}

\bibitem[Hec09]{heckenberger2009classification}
\bysame, \emph{Classification of arithmetic root systems}, Adv. Math.
  \textbf{220} (2009), no.~1, 59--124. \MR{2462836}

\bibitem[HLV12]{MR2891215}
I.~Heckenberger, A.~Lochmann, and L.~Vendramin, \emph{Braided racks, {H}urwitz
  actions and {N}ichols algebras with many cubic relations}, Transform. Groups
  \textbf{17} (2012), no.~1, 157--194. \MR{2891215}

\bibitem[HS10]{MR2732989}
I.~Heckenberger and H.-J. Schneider, \emph{Nichols algebras over groups with
  finite root system of rank two {I}}, J. Algebra \textbf{324} (2010), no.~11,
  3090--3114. \MR{2732989}

\bibitem[HV14]{MR3276225}
I.~Heckenberger and L.~Vendramin, \emph{Nichols algebras over
  groups with finite root system of rank two {II}}, J. Group Theory \textbf{17}
  (2014), no.~6, 1009--1034. \MR{3276225}

\bibitem[HV15]{MR3272075}
I.~Heckenberger and L.~Vendramin, \emph{Nichols algebras over groups with
  finite root system of rank two {III}}, J. Algebra \textbf{422} (2015),
  223--256. \MR{3272075}

\bibitem[HZ07a]{MR2336009}
Jun Hu and Yinhuo Zhang, \emph{The {$\beta$}-character algebra and a commuting
  pair in {H}opf algebras}, Algebr. Represent. Theory \textbf{10} (2007),
  no.~5, 497--516. \MR{2336009}

\bibitem[HZ07b]{MR2352888}
\bysame, \emph{Induced modules of semisimple {H}opf algebras}, Algebra Colloq.
  \textbf{14} (2007), no.~4, 571--584. \MR{2352888}

\bibitem[Kap75]{Kaplansky1975MR0435126}
I.~ Kaplansky, \emph{Bialgebras}, Lecture Notes in Mathematics, Department
  of Mathematics, University of Chicago, Chicago, Ill., 1975. \MR{0435126 (55
  \#8087)}

\bibitem[KL00]{MR1721834}
G.~R. Krause and T.~H. Lenagan, \emph{Growth of algebras and
  {G}elfand-{K}irillov dimension}, revised ed., Graduate Studies in
  Mathematics, vol.~22, American Mathematical Society, Providence, RI, 2000.
  \MR{1721834}

\bibitem[KP66]{MR0208401}
G.~I. Kac and V.~G. Paljutkin, \emph{Finite ring groups}, Trudy Moskov. Mat.
  Ob\v s\v c. \textbf{15} (1966), 224--261. \MR{0208401}

\bibitem[Mas95]{MR1357764}
A.~Masuoka, \emph{Semisimple {H}opf algebras of dimension {$6,8$}}, Israel
  J. Math. \textbf{92} (1995), no.~1-3, 361--373. \MR{1357764}

\bibitem[Mon93]{montgomery1993hopf}
S.~Montgomery, \emph{Hopf algebras and their actions on rings}, CBMS
  Regional Conference Series in Mathematics, vol.~82, Published for the
  Conference Board of the Mathematical Sciences, Washington, DC; by the
  American Mathematical Society, Providence, RI, 1993. \MR{1243637}

\bibitem[{Nic}78]{nichols1978bialgebras}
W.~D. {Nichols}, \emph{Bialgebras of type one*}, Communications in Algebra
  \textbf{6} (1978), no.~15, 1521--1552.

\bibitem[Rad85]{radford1985structure}
D.~E. Radford, \emph{The structure of {H}opf algebras with a projection}, J.
  Algebra \textbf{92} (1985), no.~2, 322--347. \MR{778452 (86k:16004)}

\bibitem[Rad03]{radford2003oriented}
\bysame, \emph{On oriented quantum algebras derived from representations of the
  quantum double of a finite-dimensional {H}opf algebra}, J. Algebra
  \textbf{270} (2003), no.~2, 670--695. \MR{2019635}

\bibitem[Sch96]{MR1396857}
P.~Schauenburg, \emph{A characterization of the {B}orel-like subalgebras of
  quantum enveloping algebras}, Comm. Algebra \textbf{24} (1996), no.~9,
  2811--2823. \MR{1396857}

\bibitem[SV12]{MR2879228}
D.~S. Sage and M.~D. Vega, \emph{Twisted {F}robenius-{S}chur indicators
  for {H}opf algebras}, J. Algebra \textbf{354} (2012), 136--147. \MR{2879228}

\bibitem[WLD16]{MR3490761}
Jinyong Wu, Gongxiang Liu, and Nanqing Ding, \emph{Classification of affine
  prime regular {H}opf algebras of {GK}-dimension one}, Adv. Math. \textbf{296}
  (2016), 1--54. \MR{3490761}

\bibitem[WZZ13]{MR3061686}
D.-G. Wang, J.~J. Zhang, and G.~Zhuang, \emph{Hopf algebras of {GK}-dimension
  two with vanishing {E}xt-group}, J. Algebra \textbf{388} (2013), 219--247.
  \MR{3061686}

\end{thebibliography}
\end{document}